\documentclass{article}
\usepackage{amssymb,amsmath,theorem,hyperref}

\nonstopmode

\long\def\killtext#1{}
\newtheorem{theorem}{Theorem}[section]
\newtheorem{prop}[theorem]{Proposition}
\newtheorem{lemma}[theorem]{Lemma}
\newtheorem{claim}{Claim}[theorem]
\newtheorem{corollary}[theorem]{Corollary}
\theorembodyfont{\rmfamily}
\newtheorem{definition}[theorem]{Definition}
\newtheorem{example}[theorem]{Example}

\newtheorem{remark}[theorem]{Remark}

\newenvironment{proof}{\medskip\noindent{\bf Proof. }}{\hfill$\square$\medskip}

\begin{document}

\addtolength{\baselineskip}{3pt} \setlength{\oddsidemargin}{0.2in}

\def\N{{\mathbb N}}
\def\R{{\mathbb R}}
\def\one{{\mathbf 1}}
\def\Q{{\mathbb Q}}
\def\Z{{\mathbb Z}}
\def\C{{\mathbb C}}
\def\hom{{\rm hom}}
\def\Hom{{\rm Hom}}
\def\Inj{{\rm Inj}}
\def\Ind{{\rm Ind}}
\def\inj{{\rm inj}}
\def\ind{{\rm ind}}
\def\sur{{\rm sur}}
\def\PAG{{\rm PAG}}
\def\eps{\varepsilon}
\def\Ge{{\mathbb G}}
\def\Ha{{\mathbb H}}
\def\CUT{\text{\rm CUT}}
\def\IO{{\infty\to1}}
\def\GR{\text{\rm GR}}
\def\irreg{\text{\rm irreg}}

\def\sto{\stackrel\square\longrightarrow}
\def\sign{\hbox{\rm sgn}}

\def\proofend{\hfill$\square$\medskip}
\def\Proof{\noindent{\bf Proof. }}

\def\AA{\mathcal{A}}\def\BB{\mathcal{B}}\def\CC{\mathcal{C}}\def\DD{\mathcal{D}}\def\EE{\mathcal{E}}\def\FF{\mathcal{F}}
\def\GG{\mathcal{G}}\def\HH{\mathcal{H}}\def\II{\mathcal{I}}\def\JJ{\mathcal{J}}\def\KK{\mathcal{K}}\def\LL{\mathcal{L}}
\def\MM{\mathcal{M}}\def\NN{\mathcal{N}}\def\OO{\mathcal{O}}\def\PP{\mathcal{P}}\def\QQ{\mathcal{Q}}\def\RR{\mathcal{R}}
\def\SS{\mathcal{S}}\def\TT{\mathcal{T}}\def\UU{\mathcal{U}}\def\VV{\mathcal{V}}\def\WW{\mathcal{W}}\def\XX{\mathcal{X}}
\def\YY{\mathcal{Y}}\def\ZZ{\mathcal{Z}}

\def\Pr{{\sf P}}
\def\E{{\sf E}}
\def\Var{{\sf Var}}
\def\T{^{\sf T}}

\title{Testing properties of graphs and functions}
\date{January 2008}
\author{{\sc L\'aszl\'o Lov\'asz}\footnote{Research supported by OTKA grant No. 67867}\\
Institute of Mathematics, E\"otv\"os Lor\'and University, Budapest\\ and\\
{\sc Bal\'azs Szegedy}\\ Department of Mathematics, University of
Toronto}

\maketitle

\tableofcontents

\bigskip\goodbreak

\begin{abstract}
We define an analytic version of the graph property testing problem,
which can be formulated as studying an unknown 2-variable symmetric
function through sampling from its domain and studying the random
graph obtained when using the function values as edge probabilities.
We give a characterization of properties testable this way, and
extend a number of results about ``large graphs'' to this setting.

These results can be applied to the original graph-theoretic property
testing. In particular, we give a new combinatorial characterization
of testable graph properties. Furthermore, we define a class of graph
properties (flexible properties) which contains all the hereditary
properties, and generalize various results of Alon, Shapira, Fischer,
Newman and Stav from hereditary to flexible properties.
\end{abstract}

\section{Introduction}

Graph property testing is a very active area in computer science. In
its most restricted form (and this will be our concern in this
paper), it studies properties of (very large) graphs that can be
tested by studying a randomly chosen induced subgraph of bounded
size.

To be more precise, we have to describe what kind of error is
allowed. In this paper, by {\it graph} we always mean a finite simple
graph. A {\it graph property} is a class of graphs invariant under
isomorphism. The {\it edit distance} of two graphs $G_1,G_2$ on the
same node set is $|E(G_1)\triangle E(G_2)|$. The {\it edit distance}
of a graph $G$ from a graph property is the minimum number of edges
we have to change (add or delete) to obtain a graph with the
property. If no graph with the same number of nodes has the property,
then this distance is infinite.

\begin{definition}\label{PROPTEST}
A graph property $\PP$ is {\bf testable}, if there exists another
property $\PP'$ (called a {\it test property}) satisfying the
following conditions:

\smallskip

(a) if a graph $G$ has property $\PP$, then for all $1\le k\le
|V(G)|$ at least a fraction of $2/3$ of its $k$-node induced
subgraphs have property $\PP'$, and

\smallskip

(b) for every $\eps>0$ there is a $k_\eps\ge 1$ such that if $G$ is a
graph whose edit distance from $\PP$ is at least $\eps |V(G)|^2$,
then for all $k_\eps\le k\le |V(G)|$ at most a fraction of $1/3$ of
the $k$-node induced subgraphs of $G$ have property $\PP'$.
\end{definition}

The notion of testability has other variations: we may also know the
number of nodes of $G$, or we can take a sample whose size is growing
slowly with the size of $G$, etc. The definition above is in a sense
the most restrictive, and it has often been referred to with
adjectives like ``oblivious testing'' and "order independent
testing". Since this is the only version we consider in this paper,
we simplify terminology by calling it simply ``testable''.

We could strengthen this definition by requiring a fraction of
$1-\eps$ instead of $2/3$ and a fraction of $\eps$ instead of $1/3$.
We could also weaken it by allowing the test property $\PP'$ to
depend on $\eps$. It can be seen that neither of these modifications
would change the notion of testability.

A surprisingly general sufficient condition for testability was
proved by Alon and Shapira \cite{AS}: {\it Every hereditary graph
property is testable.} (A graph property is {\it hereditary}, if
whenever a graph has the property, then all its induced subgraphs
also have the property.) Alon, Fischer, Newman and Shapira
\cite{AFNS} gave a characterization of testable graph properties in
terms of Szemer\'edi partitions (which is quite involved  and we
don't quote it here). In fact, Szemer\'edi partitions play a central
role in most results of this theory.

One of the main graph theoretic results in this paper is to give
another combinatorial characterization of testable properties
(Theorem \ref{COMBCHAR}). It says that a graph property is testable
if and only if for every graph with the property, a sufficiently
large ``typical'' induced subgraph is ``close'' to having the
property.

Our main goal is, however, to treat property testing in terms of the
theory of convergent (dense) graph sequences and graph limits
\cite{LSz,LSz2,BCLSV1}. A sequence of graphs $(G_n)$ is {\it
convergent} if the density of copies of any fixed graph $F$ in $G_n$
tends to a limit. It turns out that the limit of a convergent graph
sequence can be represented by a symmetric measurable function
$W:~[0,1]^2\to[0,1]$, and that many problems and constructions in
graph theory have a simpler and cleaner formulation when extended to
this limit (see \cite{BCLSV0} for a survey).

{\it Parameter testing} (or estimation) is closely related to
property testing, but is in many respects simpler. This area has a
very natural treatment in the framework of graph limits
\cite{BCLSSV,BCLSV1}. The two theories are connected by a result of
Fischer and Newman \cite{FN}, who proved that the edit distance from
a testable property is a testable parameter. The analytic theory of
property testing is more involved than the analytic theory of
parameter testing, mainly because of the different type of error that
is permitted.

Above, we used the ``edit distance'' of two graphs in the definition
of testable properties. However, there is a different distance,
called the ``cut distance'', which plays a central role in graph
convergence; for example, a sequence of graphs is convergent if and
only if it is Cauchy in an appropriately normalized cut distance. The
main technical issue in the analytic theory of property testing is
the interplay between these two distances; see Section
\ref{RELATE-NORMS} for some auxiliary results of this nature that
might be interesting on their own right.

The space of limit objects (two-variable functions) with the ``cut
distance'' is compact, a fact which is essentially equivalent to
various (weak and strong) versions of Szemer\'edi's Regularity Lemma
\cite{LSz2}. So while we do not explicitly use the Regularity Lemma,
it is implicit in the utilization of the compactness of this space.

A further surprisingly general result using the edit distance is the
theorem of Alon and Stav \cite{ASt}, proving that for every
hereditary property, a random graph with appropriate density is the
farthest from the property in edit distance. The analytic results
developed in this paper allow us to state and prove a simple analytic
analogue of this fact, from which the original result follows along
with generalizations. Similar analytic analogues are derived for the
other above mentioned results.

\section{Preliminaries}

\subsection{Homomorphisms}

For two graphs $F$ and $G$, a {\it homomorphism} from $F$ to $G$ is
an adjacency preserving map $V(F)\to V(G)$. The number of such
homomorphisms is denoted by $\hom(F,G)$. We'll almost always use the
normalized version of this number,
\[
t(F,G)=\frac{\hom(F,G)}{|V(G)|^{|V(F)|}},
\]
which can be interpreted as the probability that a random map
$V(F)\to V(G)$ is a homomorphism. We denote by $\ind(F,G)$ the number
of those injective homomorphisms that also preserve non-adjacency (in
other words, the number of induced copies of $F$ in $G$). The
normalized version of this number is
\[
t_\ind(F,G)=\frac{\ind(F,G)}{(|V(G)|)_{|V(F)|}}
\]
(here $(n)_k=n(n-1)\cdots(n-k+1)$).

\subsection{Functions and graphons}

Let $\WW$ denote the space of all symmetric measurable functions
$W:~[0,1]^2\to\R$ (i.e., $W(x,y)=W(y,x)$ for all $x,y\in[0,1]$). Let
$\WW_0$ denote the set of all functions $W\in\WW$ such that $0\le
W\le 1$.

A function $W\in\WW$ is called a {\it stepfunction}, if there is a
partition $S_1\cup\dots\cup S_k$ of $[0,1]$ into measurable sets such
that $W$ is constant on every product set $S_i\times S_j$. The number
$k$ is the {\it number of steps} of $W$.

Let $\FF_k$ denote the set of all graphs on node set
$[k]=\{1,\dots,k\}$. For $W\in\WW_0$ and $F\in\FF_k$, define
\[
t(F,W)=\int_{[0,1]^k} \prod_{ij\in E(F)} W(x_i,x_j)\,dx
\]
and
\[
t_\ind(F,W)=\int_{[0,1]^k} \prod_{ij\in E(F)}
W(x_i,x_j)\prod_{ij\notin E(F)\atop i\not= j}
\bigl(1-W(x_i,x_j)\bigr)\,dx\,.
\]

Two functions $W_1,W_2\in\WW_0$ are {\it isomorphic}, in notation
$W_1\cong W_2$, if $t(G,W_1)=t(G,W_2)$ for every graph $G$. It was
proved by Borgs, Chayes and Lov\'asz \cite{BCL} that two functions
are isomorphic if and only if there is a third function $U\in\WW_0$
and two measure preserving maps $\phi_1,\phi_2:~[0,1]\to[0,1]$ such
that $W_i(x,y)=U(\phi_i(x),\phi_i(y))$ for $i=1,2$ and almost all
$x,y\in[0,1]$. An isomorphism class of functions in $\WW_0$ is called
a {\it graphon} \cite{BCLSV1}.

A sequence of graphs $(G_n)$ with $|V(G_n)|\to\infty$ is called {\it
convergent}, if $t(F,G_n)$ tends to a limit for every fixed graph
$F$. (This is equivalent with $t_\ind(F,G_n)$ tending to a limit for
every $F$.) It was proved in \cite{LSz} that for every convergent
sequence of graphs $(G_n)$ there is a function $W\in \WW_0$ such that
$t(F,G_n)\to t(F,W)$ for every graph $F$. We call $W$ the {\it limit}
of the sequence, and write $G_n\to W$. Every function in $\WW_0$
arises as the limit of a convergent graph sequence. Furthermore, the
limit is unique up to isomorphism.

For every graph $G$, we define a function $W_G\in\WW_0$ as follows.
Let $V(G)=\{1,\dots,n\}$ and consider a point $(x,y)\in[0,1]^2$.
Define integers $i$ and $j$ such that $x\in ((i-1)/n,i/n]$ and $y\in
((j-1)/n,j/n]$ (if $x=0$ we define $i=0$, and similarly for $j$).
Then we set
\[
W(x,y)=
  \begin{cases}
    1, & \text{if $ij\in E(G)$}, \\
    0, & \text{otherwise}.
  \end{cases}
\]
(informally, we consider the adjacency matrix $A=(a_{ij})$ of $G$,
and replace each entry $a_{ij}$ by a square of size
$(1/n)\times(1/n)$ with the constant function $a_{ij}$ on this
square). Note that $W_G$ depends on the labeling of the nodes of $G$
(but only up to a measure preserving transformation).

\subsection{Distances of graphs and functions}

As mentioned in the introduction, our results concern the interaction
of two distances between graphs, the edit distance and the cut
distance. Let $G_1$ and $G_2$ be two graphs with a common node set
$V$. Instead of the edit distance mentioned in the introduction, we
shall use its normalized version
\[
d_1(G_1,G_2)=\frac{|E(G_1)\triangle E(G_2)|}{|V|^2},
\]
and the normalized cut distance
\[
d_\square(G_1,G_2)=\max_{S,T\subset V}
\frac{\Bigl|e_{G_1}(S,T)-e_{G_2}(S,T)\Bigr|}{|V|^2}
\]
(here $e_{G_1}(S,T)$ denotes the number of edges of $G$ with one
endpoint in $S$ and  the other endpoint in $T$).

We consider on $\WW$ the {\it cut norm}
\[
\|W\|_\square = \sup_{S,T\subseteq[0,1]}\Bigl|\int_{S\times T}
W(x,y)\,dx\,dy\Bigr|
\]
where the supremum is taken over all measurable subsets $S$ and $T$.
(See \cite{BCLSV1} for several useful properties of this norm.) We
will also use the standard $L_1$ norm
\[
\|W\|_1=\int_{[0,1]^2} |W(x,y)|\,dx\,dy.
\]
This defines two metrics on $\WW_0$ by
\[
d_1(U,W)=\|U-W\|_1,\qquad\text{and}\qquad
d_\square(U,W)=\|U-W\|_\square\,.
\]
For every set $S\subseteq\WW_0$ and every $c>0$, we define, as usual,
the balls
\[
B_1(S,c)=\{W\in\WW_0:~ d_1(W,S)<c\} \qquad\text{and}\qquad
B_\square(S,c)=\{W\in\WW_0:~ d_\square(W,S)<c\}.
\]
Clearly $d_\square\le d_1$, and hence $B_1(S,c)\subseteq
B_\square(S,c)$. So $d_\square$ is continuous with respect to $d_1$.
In general, $d_1$ is not continuous w.r.t. $d_\square$, but see
Theorem \ref{PROPTEST-FN} for a weaker statement.

For $W\in\WW_0$ and $\phi:~[0,1]\to[0,1]$, set
$W^\phi(x,y)=W(\phi(x),\phi(y))$. We define
\[
\delta_1(U,W)=\inf_\phi d_1(U,W^\phi),
\]
where $\phi$ ranges over all invertible measure preserving maps from
$[0,1]$ to $[0,1]$. This is a quasimetric on $\WW_0$, in which $U$
and $W$ are distance $0$ if and only if they are isomorphic. The
metric $\delta_\square(U,W)$ is defined analogously.

The main advantage of $\delta_\square$ over $d_\square$ is that the
space $(W,\delta_\square)$ is compact, as was proved in \cite{LSz2}.

We can use this distance to define yet another distance between
graphs:
\[
\delta_\square(G_1,G_2)= \delta_\square(W_{G_1},W_{G_2}).
\]
Note that in this definition the graphs $G_1$ and $G_2$ do not need
to have the same number of nodes, and their distance is independent
of the labeling of their nodes. If it happens that $V(G_1)=V(G_2)$,
then clearly
\[
\delta_\square(G_1,G_2)\le d_\square(G_1,G_2)\le d_1(G_1,G_2).
\]
The paper \cite{BCLSV1} contains a more explicit description of this
distance, and its relation to combinatorially defined distances. One
of the main conclusions is that a sequence of graphs is convergent if
and only if it is Cauchy in this metric. So $(\WW_0,\delta_\square)$
is the completion of the set of graphs with distance
$\delta_\square$.

We summarize some further facts about homomorphism densities and
distances, mostly from \cite{LSz}. Let $F$ be a graph with $k$ nodes
and $m$ edges. For every graph $G$, we have
\[
t(F,G)=t(F,W_G),
\]
but for the ``induced'' versions we only have the following
approximate equality:
\begin{equation}\label{T-IND}
\bigl|t_\ind(F,G)-t_\ind(F,W_G)\bigr| \le
\frac{\binom{k}{2}}{|V(G)|-\binom{k}{2}}.
\end{equation}
Lemma 4.1 in \cite{LSz} asserts that for any two functions
$U,W\in\WW_0$,
\begin{equation}\label{TSQUARE-FN}
|t(F,U)-t(F,W)|\le m\|U-W\|_\square.
\end{equation}
This implies (via the functions $W_G$ and $W_H$) a similar inequality
for any two graphs $G$ and $H$:
\begin{equation}\label{TSQUARE-GR}
|t(F,G)-t(F,H)|\le m\delta_\square(G,H).
\end{equation}
An analogue of inequality \eqref{TSQUARE-FN} for induced densities in
functions can be proved by essentially the same argument:
\begin{equation}\label{kozel-fn}
|t_\ind(F,U)-t_\ind(F,W)|\leq \binom{k}{2} \cdot \|U-W\|_\square.
\end{equation}
Using the easy inequality \eqref{T-IND}, this implies for the induced
densities in graphs that
\begin{equation}\label{kozel}
|t_\ind(F,H)-t_\ind(F,G)|\leq \binom{k}{2} \delta_{\square}(G,H) +
\frac{\binom{k}{2}}{|V(G)|-\binom{k}{2}}+
\frac{\binom{k}{2}}{|V(H)|-\binom{k}{2}}
\end{equation}
(assuming that $|V(G)|,|V(H)| > \binom{k}{2}$.

The following result from \cite{BCLSV1} (Theorem 4.10) provides a
converse to \eqref{TSQUARE-FN}:

\begin{theorem}\label{LEFT-CLOSE}
Let $U,W\in\WW_0$ and let $k>1$ be a positive integer. Assume that
for every simple graph $F$ on $k$ nodes, we have
\[
|t(F,U)-t(F,W)|\le 3^{-{k^2}}.
\]
Then
\[
\delta_\square(U, W) \le
{\frac{22}{\sqrt{\log k}}.}
\]
\end{theorem}

We conclude with a lemma showing that convergence in the cut norm has
good analytic properties.

\begin{lemma}\label{WSTAR}
Suppose that $\|W_n-W\|_\square\to 0$ as $n\to\infty$
($W,W_n\in\WW_0$). Then for every $Z\in\WW_0$
\[
\|Z(W_n-W)\|_\square \to 0 \qquad (n\to \infty).
\]
In particular,
\[
\int_{[0,1]^2} Z(x,y)W_n(x,y)\,dx\,dy \to \int_{[0,1]^2}
Z(x,y)W_n(x,y)\,dx\,dy
\]
and
\[
\int_S W_n \to \int_S W
\]
for every measurable set $S\subseteq [0,1]^2$.
\end{lemma}

\begin{proof}
If $Z$ is the indicator function of a rectangle, these conclusions
follow from the definition of the $\|.\|_\square$ norm. Hence the
conclusion follows for stepfunctions, since they are linear
combinations of a finite number of indicator functions of rectangles.
Then it follows for all integrable functions, since they are
approximable in $L_1([0,1]^2)$ by stepfunctions.
\end{proof}

\subsection{$W$-random graphs}\label{SEC:WRAND}

For every function $W\in \WW_0$ and every finite set $X\subseteq
[0,1]$, we define a graph $\Ge(X,U)$ on $V(G)=X$ by connecting
$x,y\in X$, $x\not= y$ with probability $U(x,y)$ (making independent
decisions for different pairs $\{x,y\}$). If $G$ is a simple graph
and $X\subseteq V(G)$, we denote by $G[X]$ the subgraph induced by
$X$.

For every function $W\in \WW_0$, we define the {\it $W$-random graph}
$\Ge(n,W)=\Ge(X,W)$, where $X\subseteq[0,1]$ consists of $n$
independent, uniformly distributed random points in $[0,1]$. Note
that for every $F\in\FF_n$,
\[
\Pr(\Ge(n,W)=F)=t_\ind(F,W).
\]
Clearly the distribution of $\Ge(n,W)$ is invariant under isomorphism
of functions, i.e., it only depends on $W$ as a graphon.

We need an analogous notation $\Ge(k,G)=G[X]$, where $G$ is a finite
graph and $X$ is a random subset of $V(G)$ chosen uniformly from all
$k$-element subsets.

It was proved in \cite{LSz} and with probability $1$, $\Ge(n,W)\to
W$, and in fact, the convergence is quite fast, as shown by the
following concentration inequality:

\begin{theorem}\label{CONC}
Let $W\in\WW_0$ and let $F$ be a graph with $k$ nodes. Then for every
$0<\eps<1$,
\begin{equation}\label{TFH}
\Pr\Bigl(|t(F,\Ge(n,W))-t(F,W)| > \eps \Bigr) \le
2\exp\left(-\frac{\eps^2}{18k^2}n\right).
\end{equation}
\end{theorem}

The following bound on the distance of a $W$-random graph from $W$
was proved in \cite{BCLSV1} (Theorem 4.9(ii)):

\begin{theorem}\label{THM:SAMPLE2}
Let $U\in\WW_0$ and let $k>1$ be a positive integer. Then with
probability at least $1-e^{-k^2/(2\log k)}$, we have
\[
\delta_\square(U, \Ge(k,U))
\le \frac{10}{\sqrt{\log k}}.
\]
\end{theorem}

We'll use a related construction of a random graph associated with
$W$. Let $X_i$ be a uniform random element of the interval
$L_i=[(i-1)/n,i/n]$, and $X=\{X_1,\dots,X_n\}$. Let
$\Ge'(n,W)=\Ge(X,W)$. (Note that the distribution of $\Ge(n,W)$ only
depends on the isomorphism type of $W$, but $\Ge'(n,W)$ does depend
on how the points of $[0,1]$ are ordered.) For the cut distance of
$\Ge'(n,W)$ from $W$ we have the following bound:

\begin{lemma}\label{LEM:GPRIME}
Let $W\in\WW$ and $G'=\Ge'(n,W)$. Then with probability more than
$1-1/\sqrt{n}$,
\[
\delta_\square(W_{G'},U) < \frac{50}{\sqrt{\log\log n}}.
\]
\end{lemma}

\begin{proof}
We use subgraph densities and Theorem \ref{LEFT-CLOSE}. To estimate
$|t(F,W_{\Ge(n,U)})-t(F,W_{G'})|$, note that the random variable
$t(F,W_{\Ge(n,U)})$ can be generated as follows: (1) we select $k$
random integers $r_1,\dots,r_k\in[n]$ uniformly and independently
(with repetition); (2) we select $Z_i\in [0,1/n]$ uniformly and
independently for $i=1,\dots,k$, and compute $X_i=Z_i+(r_i-1)/n$; (3)
we create random variables $Y_{ij}$ ($i,j\in[k]$) such that
$Y_{ij}=Y_{ji}$ but otherwise they are independent, and
\[
Y_{ij}=
  \begin{cases}
    1 & \text{with probability $W(X_i,X_j)$}, \\
    0 & \text{with probability $1-W(X_i,X_j)$}.
  \end{cases}
\]
Then $t(F,W_{\Ge(n,U)})$ is the expectation of $\prod_{ij\in
E(F)}Y_{ij}$ over the choice (1). The computation of $t(F,W_{G'})$ is
similar, except that if $r_i=r_j$ then we choose $Z_i=Z_j$. It
follows that if the integers $r_1,\dots,r_k$ are distinct, then
$t(F,W_{\Ge(n,U)})$ and $t(F,W_{G'})$ are generated by the same
procedure, and hence they can be coupled so that they are equal. The
probability that there is no repetition among $r_1,\dots,r_k$ is
\[
\frac{n(n-1)\cdots(n-k+1)}{n^k} > 1- \frac{\binom{k}{2}}{n},
\]
So it follows that $t(F,W_\Ge(n,U))=t(F,W_{G'})$ with probability at
least $1- \frac{\binom{k}{2}}{n}$, and hence
\[
\E\bigl(|t(F,W_{\Ge(n,U)})-t(F,W_{G'})|\bigr)<\frac{k^2}{n}.
\]
Summing this over all graphs $F$ on $k$ nodes, we get
\[
\E\Bigl(\sum_F |t(F,W_{\Ge(n,U)})-t(F,W_{G'})|\Bigr) <
\frac{2^{k^2}}{n},
\]
and so the probability that $\sum_F |t(F,W_{\Ge(n,U)})-t(F,W_{G'})| >
3^{-k^2}$ is less than $6^{k^2}/n$. Thus with probability more than
$1-6^{k^2}/n$, we have
\[
|t(F,W_{\Ge(n,U)})-t(F,W_{G'})\Bigr)| <3^{-k^2}
\]
for all graphs $F$ with at most $k$ nodes. Theorem \ref{LEFT-CLOSE}
implies that in this case
\[
\delta_\square(\Ge(n,U),G') \le \frac{22}{\sqrt{\log k}}.
\]
So if we choose $k=\lfloor\sqrt{\log (n/2)}/4\rfloor$, then
\[
\delta_\square(\Ge(n,U), G') \le \frac{44}{\sqrt{\log\log n}}.
\]
Combining with \eqref{EQ:DSGU}, we get that with probability more
than $1-1/\sqrt{n}$,
\[
\delta_\square(W_{G'},U)\le \delta_\square(G', \Ge(n,U))
+ \delta_\square(W_{\Ge(n,U)},U)\le \frac{10}{\sqrt{\log n}} +
\frac{44}{\sqrt{\log\log n}} < \frac{50}{\sqrt{\log\log n}}.
\]
\end{proof}

One advantage of using $\Ge'(n,W)$ over $\Ge(n,W)$ is that it is
easier to handle its $d_1$ distance from other graphs and functions.

\begin{lemma}\label{LEM:GPRIMEL1}
Let $W\in\WW$ and $G'=\Ge'(n,W)$. Let $H$ be a graph on $[n]$. Then
\[
\E(d_1(G',H))= \|W_H-W\|_1
\]
and
\[
\E(\|W_{G'}-W\|_1)\le 2\|W_H-W\|_1.
\]
\end{lemma}

\begin{proof}
For the first formula, note that for $1\le i,j\le n$, the probability
that a pair $i\not=j$ contributes to $d_1(H,G')$ is
\[
  \begin{cases}
    W(X_i,X_j), & \text{if $ij\notin E(H)$}, \\
    1-W(X_i,X_j), & \text{if $ij\in E(G_n)$}.
  \end{cases}
\]
Summing over all $i\not=j$, and taking expectation, the equality
follows. The second inequality is an easy consequence:
\[
\E(\|W_{G'}-W\|_1) \le \E(\|W_{G'}-W_H\|_1)+\|W_H-W\|_1
=\E(d_1(G',H))+\|W_H-W\|_1\le 2\|W_H-W\|_1.
\]
\end{proof}

As a useful consequence, we obtain the following fact. The Regularity
Lemma implies (see e.g. \cite{LSz2}) that functions in $\WW_0$ can be
approximated by functions of the form $W_G$, so that the number of
nodes of $G$ can be bounded uniformly if the error is measured in the
cut distance. Obviously, one cannot approximate all functions by
functions $W_G$ in the $L_1$-norm. But for every $n$ there is graph
on $n$ nodes that approximates $W$ so that the cut distance tends to
$0$ uniformly, and at the same time the approximation in the $L_1$
norm is almost as good as possible.

\begin{corollary}\label{COR:SQUARECLOSE}
Let $G$ be a simple graph and $U\in\WW_0$. Then there exists a simple
graph $G'$ such that
\[
\|W_{G'}-U\|_1\le 4\|W_G-U\|_1\qquad\text{and}\qquad
\delta_\square(W_{G'},U)\le \frac{50}{\sqrt{\log\log n}}.
\]
\end{corollary}

\begin{proof}
Let $V(G)=[n]$. We are going to show that the random graph
$G'=\Ge'(n,W)$ satisfies the conditions with positive probability. By
Lemma \ref{LEM:GPRIMEL1}, $\E\bigl(d_1(W_G',W)\bigr) \le
2\|W_G-W\|_1$, and so with probability at least $1/2$,
\begin{equation}\label{EQ:WGEW}
d_1(W_{G'},W)\le 4\|W_G-W\|_1.
\end{equation}
On the other hand, Lemma \ref{LEM:GPRIME} implies that with
probability at least $3/4$, we have
\begin{equation}\label{EQ:DSGU}
\delta_\square(G',U) \le \frac{50}{\sqrt{\log\log n}}.
\end{equation}
So with positive probability, both \eqref{EQ:WGEW} and
\eqref{EQ:DSGU} hold.
\end{proof}

\subsection{Relating different norms}\label{RELATE-NORMS}

As pointed out in the introduction, the analytic problem behind
property testing is to relate the cut norm and the $L_1$ norm. In
this section, which contains our main technical tools, we study this
connection. Some of the lemmas below are purely analytic in nature,
and may be of interest on their own right.

\begin{lemma}\label{fura1}
Suppose that $\|U_n-U\|_\square\to 0$ and $\|W_n-W\|_\square\to 0$ as
$n\to\infty$ ($U,W,U_n,W_n\in\WW_0$). Then
\[
\liminf_{n\to\infty} \|W_n-U_n\|_1\ge \|W-U\|_1.
\]
\end{lemma}

\begin{proof}
Let $\sigma(x,y)=\sign(U(x,y)-W(x,u))$. Then
\begin{align*}
\|U_n-W_n\|_1 & \ge  \int_{[0,1]^2}
\sigma(x,y)(U_n(x,y)-W_n(x,y))\,dx\,dy\\
&\longrightarrow\int_{[0,1]^2} \sigma(x,y)(U(x,y)-W(x,y))\,dx\,dy
=\|U-W\|_1
\end{align*}
by Lemma \ref{WSTAR}.
\end{proof}

There are easy examples showing that $\|W_n-U_n\|_1\to \|W-U\|_1$
does not hold in general: for example, let $W_n=W_{\Ge(n,1/2)}$ and
$W=U=U_n=1/2$. But we can formulate two statements that provide
partial converses to this lemma.

\begin{lemma}\label{REV-CONV}
Suppose that $\|U_n-U\|_\square\to 0$ and $\|W_n-W\|_\square\to 0$ as
$n\to\infty$ ($U,W,U_n,W_n\in\WW_0$). Suppose further that $U$ is
$0-1$ valued. Then
\[
\|U_n-W_n\|_1 \to \|U-W\|_1.
\]
\end{lemma}

\begin{proof}
By Lemma \ref{WSTAR},
\begin{align*}
\|U_n-W_n\|_1&\le \|U_n-U\|_1+\|U-W_n\|_1\\
&=\int_{U=0} U_n +\int_{U=1} (1-U_n)+ \int_{U=0} W_n +\int_{U=1} (1-W_n)\\
&\to\int_{U=0} W +\int_{U=1} (1-W)=\|U-W\|_1\,.
\end{align*}
Combined with Lemma \ref{fura1}, the assertion follows.
\end{proof}

\begin{lemma}\label{REV-CONV-1}
Suppose that $\|U_n-U\|_\square\to 0$ as $n\to\infty$
($U,U_n\in\WW_0$). Then for every $W\in\WW_0$ there is a sequence of
functions $W_n\in \WW_0$ such that $\|W_n-W\|_\square\to 0$ and
\[
\|U_n-W_n\|_1 \to \|U-W\|_1.
\]
\end{lemma}

\begin{proof}
First we consider the case when $U\ge W$. Let
\[
Z(x,y)=
  \begin{cases}
   W(x,y)/U(x,y)  & \text{if $U(x,y)>0$}, \\
    0 & \text{otherwise}.
  \end{cases}
\]
and $W_n=Z U_n$. Trivially $W_n\in\WW_0$, $W=ZU$, and
\[
\|W-W_n\|_\square = \|Z(U-U_n)\|_\square \to 0
\]
by Lemma \ref{WSTAR}. Furthermore,
\[
\|U_n-W_n\|_1 = \|(U-W) + Z(U-U_n)\|_1 \le \|U-W\|_1 +
\|Z(U-U_n)\|_1\to \|U-W\|_1
\]
(using Lemma \ref{WSTAR} again). Combining this with Lemma
\ref{fura1} we get that $\|U_n-W_n\|_1 \to \|U-W\|_1$.

The case when $U\le W$ follows by a similar argument, replacing
$U,W,\dots$ by $1-U, 1-W,\dots$.

Finally, in the general case, consider the function $V=\max(U,W)$.
Then clearly $\|U-V\|_1+\|V-W\|_1 = \|U-W\|_1$. Since $U\le V$, there
exists a sequence $(V_n)$ of functions such that
$\|V_n-V\|_\square\to 0$ and $\|V_n-U_n\|_1\to \|V-U\|_1$. Since
$V\ge W$, there is a sequence $(W_n)$ of functions such that
$\|W_n-W\|_\square\to 0$ and $\|W_n-V_n\|_1\to \|W-V\|_1$. Hence
\begin{align*}
\limsup_{n\to\infty} \|U_n-W_n\|_1 &\le\limsup_{n\to\infty}
\|U_n-V_n\|_1 +\limsup_{n\to\infty} \|V_n-W_n\|_1 \\
&= \|U-V\|_1 + \|V-W\|_1 = \|U-W\|_1.
\end{align*}
Using Lemma \ref{fura1} again, the lemma follows.
\end{proof}

A fact similar to Lemma \ref{fura1} holds for the distances
$\delta_1$ and $\delta_\square$ replacing the norms $\|.\|_1$ and
$\|.\|_\square$. This does not seem to follow directly from Lemma
\ref{fura1}, and the proof is more involved.

\begin{lemma}\label{fura}
Suppose that $\delta_\square(U_n,U)\to 0$ and
$\delta_\square(W_n,W)\to 0$ as $n\to\infty$ $(U,W,U_n,W_n\in\WW_0)$.
Then
\begin{equation}\label{EQ:SQU1}
\liminf \delta_1(W_n,U_n)\ge \delta_1(W,U).
\end{equation}
\end{lemma}

\begin{proof}
Let $d$ denote the lim inf on the left hand side of \eqref{EQ:SQU1}.
Let $\eps$ be an arbitrary positive number. There is a number $k$
(depending on $U$, $W$ and $\eps$), a partition $[0,1]=\cup_{i=1}^k
S_i$ of the unit interval into $k$ measurable pieces and two
functions $W',U'\in\WW_0$ such that both $W'$ and $U'$ are constant
on every rectangle $S_i\times S_j$ and furthermore that $\|W-W'\|_1,
~ \|U-U'\|_1\leq\eps$. Let $n$ be a natural number such that $\|
W_n-W\|_\square ~,~ \| U_n-U\|_\square\leq\eps/k^4$ and
$\delta_1(W_n,U_n)\le d+\eps$. There are measure preserving
transformations $\rho,\pi:[0,1]\mapsto [0,1]$ such that
$\|W_n^{\rho}-U_n^{\pi}\|_1\leq d+2\eps$. Let $S_{i,j}$ denote the
set $S_i^{\rho}\cap S_j^{\pi}$ for every $1\leq i,j\leq k$. It is
clear that $S_{i,j}$ is a partition of the unit interval and that
both ${W'}^{\rho}$ and ${U'}^{\pi}$ are constant on each rectangle
$S_{i_1,j_1}\times S_{i_2,j_2}$. We have that
\[
d+2\eps\ge \| W_n^{\rho}-U_n^{\pi}\|_1\ge \sum_{1\leq
i_1,j_1,i_2,j_2\leq k}\Bigl|\int_{S_{i_1,j_1} \times
S_{i_2,j_2}}(W_n^{\rho}-U_n^{\pi})\Bigr|.
\]
Using that
\[
\Bigl|\int_{S_{i_1,j_1}\times S_{i_2,j_2}} (W^{\rho}-W_n^{\rho})
\Bigr|\le \frac{\eps}{k^4}
~~\text{and}~~\Bigl|\int_{S_{i_1,j_1}\times S_{i_2,j_2}}
(U^{\pi}-U_n^{\pi})\Bigr|\le \frac{\eps}{k^4},
\]
we get that
\[
\sum_{1\leq i_1,j_1,i_2,j_2\leq k}\Bigl|\int_{S_{i_1,j_1}\times
S_{i_2,j_2}}(W^{\rho}-U^{\pi})\Bigr|\le d+4\eps.
\]
Writing
$W^{\rho}-U^{\pi}=(W^{\rho}-{W'}^{\rho})+({W'}^{\rho}-U'^{\pi})
+(U'^{\pi}-U^{\pi})$ we obtain that
\begin{align*}
\sum_{1\leq i_1,j_1,i_2,j_2\leq k}&\Bigl|\int_{S_{i_1,j_1}\times
S_{i_2,j_2}}(W^{\rho}-U^{\pi})\Bigr|\geq \sum_{1\leq
i_1,j_1,i_2,j_2\leq k}\Bigl(\Bigl|\int_{S_{i_1,j_1}\times
S_{i_2,j_2}} ({W'}^{\rho}-{U'}^{\pi})\Bigl|\\&-
\Bigl|\int_{S_{i_1,j_1}\times S_{i_2,j_2}}
(W^{\rho}-{W'}^{\rho})\Bigr|-\Bigl|\int_{S_{i_1,j_1}\times
S_{i_2,j_2}} (U^{\pi}-{U'}^{\pi})\Bigr|\Bigr).
\end{align*}
This implies that
\[
\sum_{1\leq i_1,j_1,i_2,j_2\leq k}\Bigl|\int_{S_{i_1,j_1}\times
S_{i_2,j_2}}(W^{\rho}-U^{\pi})\Bigr|\ge \sum_{1\leq
i_1,j_1,i_2,j_2\leq k}\Bigl|\int_{S_{i_1,j_1}\times S_{i_2,j_2}}
({W'}^{\rho}-{U'}^{\pi})\Bigl|-2\eps.
\]
Using that both ${W'}^{\rho}$ and ${U'}^{\pi}$ are constant on the
sets $S_{i_1,j_1}\times S_{i_2,j_2}$ we obtain that the right side of
the above inequality is equal to
$\|{W'}^{\rho}-{U'}^{\pi}\|_1-2\eps$. Consequently
\[
\| {W'}^{\rho}-U'^{\pi}\|_1\leq d+6\eps.
\]
Using that $\| W^{\rho}-{W'}^{\rho}\|_1,~\|U^{\pi}-{U'}^{\pi}\|_1\le
\eps$ we get that
\[
\|W^{\rho}-U^{\pi}\|_1\leq d+10\eps.
\]
Since $\eps>0$ is arbitrary, \eqref{EQ:SQU1} follows.
\end{proof}

\section{Main results}

\subsection{Property testing for functions}

We define a notion of testability for properties of functions in
$\WW_0$ and for graphons. Formally, a {\it function property} is a
subset $\RR\subseteq\WW_0$; a {\it graphon property} is a function
property that is invariant under isomorphism. A function property is
{\it closed} if it is closed in the $\|.\|_\square$ norm.

\begin{example}\label{EXA:01}
The function property of being 0-1 valued is a graphon property by
the characterization of isomorphism, but it is not closed, since for
a sequence of random graphs $G_n$ with edge probability $1/2$, the
functions $W_{G_n}$ are 0-1 valued, but their limit in the
$\|.\|_\square$ norm, namely the identically $1/2$ function, is not.
\end{example}

The following definition of testability of a function property is
analogous to the testability of a graph property. The framework is
that we study a function $W\in\WW_0$ by observing a $W$-random
graph $\Ge(k,W)$.

\begin{definition}
A function property $\RR$ is {\it testable} if there is a graph
property $\RR'$ (called a {\it test property} for $\RR$) such that

\smallskip

(a) $\Pr(\Ge(k,W)\in\RR') \ge 2/3$ for every function $W\in\RR$ and
every $k\ge 1$, and

\smallskip

(b) for every $\eps>0$ there is a $k_\eps\ge 1$ such that
$\Pr(\Ge(k,W)\in\RR') \le 1/3$ for every $k\ge k_\eps$ and every
function $W\in\WW_0$ with $d_1(W,\RR)\ge \eps$.
\end{definition}

Similarly as for graph properties, the constants $1/3$ and $2/3$ are
arbitrary, but it would not change the property if we replaced them
by any two real numbers $0<a<b<1$:

\begin{lemma}\label{CONSTANTS}
Let $0<a<b<1$. A graphon property $\RR$ is testable if and only if
there is a graph property $\RR''$ such that for every $\eps>0$ there
is a constant $k(\eps)$ such that for every function $W\in\WW_0$ and
$k\ge k(\eps)$,

\smallskip

(a) if $W\in\RR$, then $\Pr(\Ge(k,W)\in\RR'') \ge b$, and

\smallskip

(b) if $d_1(W,\RR)>\eps$ then $\Pr(\Ge(k,W)\in\RR'') \le a$.
\end{lemma}

\begin{proof}
Suppose that $\RR$ is testable, and let $\RR'$ be the graph property
in the definition of testability. For every simple graph $F$ and
$k\le |V(F)|$, define
\[
q_k(F)=\Pr(\Ge(k,F)\in\RR').
\]
Define the graph property
\[
\RR''=\{G:~q_k(F) \ge \frac12\}.
\]
Let $W\in\RR$ and let $n>k$ be large enough. Then
\[
\E\bigl(q_k(\Ge(n,W))\bigr)=\Pr(\Ge(k,W)\in\RR') \ge \frac23.
\]
We use Azuma's inequality to show that $q_k(\Ge(n,W))$ is highly
concentrated. Indeed, let $X_1,\dots,X_n\in[0,1]$ be independent
uniform samples, and let $Z_t=\E_t\bigl(q_k(\Ge(n,W))\bigr)$, where
$E_t$ means conditional expectation with respect to the choice of
$X_1,\dots,X_t$ and the randomization of the edges between nodes
$\{1,\dots,t\}$. Then $Z_0,\dots,Z_n$ is a martingale, with
expectation $\E(Z_0)=\E\bigl(q_k(\Ge(n,W))\bigr)\ge 2/3$.
Furthermore, $|Z_{t+1}-Z_t|$ is bounded by the probability that a
random $k$-subset of $\{1,\dots,n\}$ contains $t+1$, which is $k/n$.
Thus by Azuma's Inequality,
\[
\Pr\Bigl(q_k(\Ge(n,W)) \le \frac12\Bigr)\le \Pr\Bigl(q_k(\Ge(n,W))\le
\E(q_k(\Ge(n,W)))-\frac16\Bigr) \le \exp\Bigl(-\frac{n}{72k^2}\Bigr).
\]
Choosing $n$ large enough, this probability will be less than $1-b$,
and hence
\[
\Pr\bigl(\Ge(n,W)\in\RR''\bigr) >b.
\]
So (a) is satisfied. The proof of (b) is analogous, and so is the
proof of the converse.
\end{proof}

It follows from the definition that a graphon property is testable if
and only if its closure in the $\|.\|_1$ norm is testable.
Furthermore, the closure in the $\|.\|_\square$ norm of a testable
property is testable (but not the other way around, see example
\ref{TEST-FN-EX0}(d) below). It follows from Theorem
\ref{PROPTEST-FN} below (but it is not hard to see directly too) that
if $\RR$ is testable, then its closures in the $\|.\|_1$ norm and
$\|.\|_\square$ norm coincide.

Since the distribution of $\Ge(k,W)$ is preserved under isomorphism
of $W$, every closed testable function property is a graphon
property. In our applications to graph theory only closed graphon
properties will play a role, and we'll focus on characterizing
testability for such properties.

It is trivial that $B_1(\RR,\eps)\subseteq B_\square(\RR,\eps)$ for
every $\RR\subseteq\WW$ and $\eps>0$. A reverse containment
characterizes testable graphon properties.

\begin{theorem}\label{PROPTEST-FN}
A graphon property $\RR$ is testable if and only if for every
$\eps>0$ there is an $\eps'>0$ such that $B_{\square}(\RR,\eps')
\subseteq B_1(\RR,\eps)$.
\end{theorem}

\begin{proof}
Suppose that $\RR$ is testable with test property $\RR'$. Let
$\eps>0$, let $k=k(\eps)$ be the constant in the definition, and let
$\eps'=2^{-k^2}$. Suppose that for some $W\in\WW_0$, we have
$d_\square(W,\RR)<\eps'$. Then there is a $U\in\RR$ such that
$d_\square(W,U)<\eps'$. By \eqref{kozel-fn}, we have for every graph
$F$ on $k$ nodes
\[
|t_\ind(F,W)-t_\ind(F,U)|\le \binom{k}{2}\eps'.
\]
Thus
\[
|\Pr(\Ge(k,W)\notin\RR')-\Pr(\Ge(k,U)\notin\RR')| \le
\sum_{F\in\FF_n\setminus\RR'} |t_\ind(F,W)-t_\ind(F,U)|\le
2^{\binom{k}{2}}\binom{k}{2}\eps'<\frac13.
\]
Hence $\Pr(\Ge(n,W)\notin\RR') <\Pr(\Ge(n,U)\notin\RR')
+\frac13\le\frac23$, which proves that $d_1(W,\RR)\le\eps$.

Conversely, suppose that $\RR$ satisfies the condition in the
proposition. Choose
\[
k_n=\frac12\sqrt{\log n}, \qquad \eps_n= n^{-1/3}.
\]
Let $\RR'$ be the graph property that a graph $G$ with $|V(G)|=n$ has
if and only if there exists a $U\in\RR$ such that
$|t(F,U)-t(F,G)|\le\eps_n$ for all graphs $F$ with $|V(F)|\le k_n$.
We show that (a) and (b) are satisfied.

First, suppose that $W\in\RR$. Let $n$ be large, and let $F$ be a
graph with at most $k_n$ nodes. By Theorem \ref{CONC}, we have
$|t(F,\Ge(n,W))-t(F,W)|\le \eps_n$ with probability at least
\[
1-2\exp\bigl(-\frac{n\eps_n^2}{18k_n^2}\bigr).
\]
Since the number of graphs on at most $k_n$ nodes is at most
$\exp(k_n^2/2)$, the probability that $|t(F,\Ge(n,W))-t(F,W)|\le
\eps_n$ holds for every graph $F$ on at most $k_n$ nodes is at least
\[
1-2\exp\bigl(\frac{k_n^2}2-\frac{n\eps_n^2}{18k_n^2}\bigr)>\frac23
\]
if $n$ is large enough. Since $W\in\RR$, in every such case
$\Ge(n,W)\in\RR'$. This proves that $\RR'$ satisfies (a).

Second, let $\eps>0$ and suppose that $d_1(W,\RR)>\eps$. By
hypothesis, there is an $\eps'>0$ (depending only on $\eps$) such
that $d_\square(W,\RR)>\eps'$.

Let $n$ be large, and suppose, by way of contradiction, that with
probability larger than $1/3$, we have $\Ge(n,W)\in\RR'$. In every
such case, there exists a function $U\in\RR$ such that
\begin{equation}\label{FUFW-1}
|t(F,U)-t(F,\Ge(n,W))|\le \eps_n
\end{equation}
for all graphs $F$ with at most $k_n$ nodes. Similarly as above,
Theorem 2.5 in \cite{LSz} implies that
\begin{equation}\label{FUFW-2}
|t(F,\Ge(n,W))-t(F,W)| \le \eps_n
\end{equation}
for all $F$ with at most $k_n$ nodes with probability at least $2/3$.
There is an outcome for $\Ge(n,W)$ for which both \eqref{FUFW-1} and
\eqref{FUFW-2} occur, and so there always exists a function $U\in\RR$
such that
\[
|t(F,W)-t(F,U)|\le 2\eps_n
\]
for all graphs $F$ with at most $k_n$ nodes. By Theorem 3.6 in
\cite{BCLSV1}, it follows that
\[
d_\square(U,W) \le \frac{22}{\sqrt{\log_2k_n}}.
\]
If $n$ is large enough, this is less that $\eps'$, contradicting the
definition of $\eps'$.
\end{proof}

We also prove the following characterization, which is a functional
analogue of the characterization of Alon, Fischer, Newman and Shapira
\cite{AFNS}.

\begin{theorem}\label{REG-TEST}
A graphon property $\RR$ is testable if and only if for every
$\eps>0$ there is an $\eps'>0$ and a finite set $S$ of stepfunctions
such that $\RR\subseteq B_\square(S,\eps')\subseteq B_1(\RR,\eps)$.
\end{theorem}

This theorem gives a ``constructive'' method of testing a testable
graphon property: for every fixed error bound, it suffices to compute
the $d_\square$ distance from a finite number of stepfunctions to
separate the case when $W\in\RR$ from the case when $d_1(W,\RR)\ge
\eps$.

\begin{proof}
First, suppose that $\RR$ is testable, and let $\eps>0$. By Theorem
\ref{PROPTEST-FN}, there is an $\eps'>0$ such that
$B_\square(\RR,2\eps')\subseteq B_1(\RR,\eps)$. For every
stepfunction $s$, consider the open ball $B_\square(s,\eps')$. These
balls cover the whole space, so by the compactness of
$(\WW,\delta_\square)$ there is a finite set $S_0$ of stepfunctions
such that the corresponding balls cover the whole space. Let $S$ be
the set of those stepfunctions in $S_0$ for which the corresponding
balls intersect $\RR$. Then clearly $\RR\subseteq
B_\square(S,\eps')$. On the other hand, $B_\square(U,\eps')$
intersects $\RR$ for every $U\in S$, and hence
$B_\square(S,\eps')\subseteq B_\square(\RR,2\eps')\subseteq
B_1(\RR,\eps)$.

Second, suppose that $\RR$ satisfies the condition in the theorem,
then for every $\eps>0$ there exists an $\eps'>0$ and a finite set
$S$ of stepfunctions such that $\RR\subseteq
B_\square(S,\eps')\subseteq B_1(\RR,\eps/2)$. Let
$\eps''=\eps\eps'/6$. We claim that
\begin{equation}\label{BSB1}
B_\square(S,\eps'+\eps'')\subseteq B_1(B_\square(S,\eps'),\eps/2).
\end{equation}
Indeed, let $W\in B_\square(\SS,\eps'+\eps'')$. Then there is a $U\in
\SS$ such that $\|U-W\|_\square<\eps'+2\eps''$. Consider
$Y=(1-\frac13\eps)W+\frac13\eps U$. Then
\[
\|Y-U\|_\square = \|(1-\frac13\eps)(U-W)\|_\square <
(1-\frac13\eps)(\eps'+2\eps'') < \eps',
\]
so $Y\in B_\square(S,\eps')$. On the other hand,
\[
\|W-Y\|_1 =\|\frac13\eps(W-U)\|_1 \le\frac13\eps<\frac12\eps,
\]
and so $W\in B_1(B_\square(S,\eps'),\eps/2)$. This proves
\eqref{BSB1}, which in turn implies that
\[
B_\square(\RR,\eps'')\subseteq
B_\square(B_\square(S,\eps'),\eps'')\subseteq
B_1(B_\square(S,\eps'),\eps/2)\subseteq B_1(B_1(\RR,\eps/2),\eps/2)=
B_1(\RR,\eps).
\]
This proves that $\RR$ is testable.
\end{proof}

We conclude this section with some examples of testable and
non-testable function properties.

\begin{example}\label{TEST-FN-EX0}
(a) Let $\RR=\{U\}$, where $U\in\WW$ is the identically $1/2$
function. Clearly this is invariant under isomorphism. Consider a
random graph $G_n=\Ge(n,1/2)$; then $\|W_{G_n}-U\|_\square\to 0$ with
probability $1$, but $\|W_{G_n}-U\|_1=1/2$ for every $n$. So this
property is not testable by Theorem \ref{PROPTEST-FN}.

On the other hand, the complementary property $\RR^c=\WW_0\setminus
\{U\}$ is testable; indeed, its closure is $\WW_0$ (either in the
$\|.\|_\square$ norm or the $\|.\|_1$ norm), which is trivially
testable.

\smallskip

(b) Let $\SS\subseteq\WW_0$ be an arbitrary graphon property and let
$a>0$ be an arbitrary number. Then $\RR=B_\square(\SS, a)$ is
testable.

Indeed, for $\eps>0$ define $\eps'=a\eps/(2-2\eps)$. Let $W\in
B_\square(\RR,\eps')$. Then $W\in B_\square(\SS,a+\eps')$, and so
there is a $U\in \SS$ such that $\|U-W\|_\square\le a+2\eps'$.
Consider $Y=(1-\eps)W+\eps U$. Then $\|Y-U\|_\square =
\|(1-\eps)(U-W)\|_\square \le (1-\eps)(a+2\eps') = a$, so $Y\in \RR$.
On the other hand, $\|W-Y\|_1 =\|\eps(W-U)\|_1 \le\eps$, and so $W\in
B_1(\RR,\eps)$.

\smallskip

(c) For every fixed graph $F$ and $0< c< 1$, the property $\RR$ that
$t(F,W)=c$ is testable; an appropriate test property is
\begin{equation}\label{T-TEST}
\RR'=\Bigl\{G:~|t(F,G)-c|\le \frac1{\log |V(G)|}\Bigr\}.
\end{equation}
Indeed, it follows from Theorem \ref{CONC} that with probability
$1-o(1)$,
\[
|t(F,W)-t(F,\Ge(n,W))|\le \frac1{\log n}.
\]
So if $t(F,W)=c$, then $\Ge(n,W)\in\RR'$ with large probability.
Conversely, assume that $d_1(W,\RR)>\eps$, and let (say) $t(F,W)>c$.
The functions $U_s=(1-s)W$, $0\le s\le\eps$, are all in
$B_1(W,\eps)$, and hence not in $\RR$. It follows that $t(F,U_s)>c$
for all $s$. On the other hand,
$t(F,U_\eps)=(1-\eps)^{|E(F)|}t(F,W)$, and so
$t(F,W)>(1-\eps)^{-|E(F)|}c$, which implies that $\Ge(n,W)\notin\RR'$
with large probability.

Fixing two subgraph densities, however, may yield a non-testable
property: for example, $t(K_2,W)=1/2$ and $t(C_4,W)=1/16$ imply that
$W\equiv 1/2$ (see \cite{BCLSV0}).

\smallskip

(d) The graphon property that $W$ is 0-1 valued is not testable. It
is closed in the $\|.\|_1$ norm, but its closure in the
$\|.\|_\square$ norm is the whole set $\WW_0$.
\end{example}

\subsection{Graph properties vs. function properties}

\subsubsection{Closure of graph properties}

We want to establish the connection between testability of graph
properties and graphon properties. The fact that graphons arise as
limits of graph sequences suggests the following definition.

\begin{definition}
If $\PP$ is a graph property, then we define its {\it closure}
$\overline\PP$ as the set of all functions $W\in\WW_0$ for which
there exists a sequence of graphs $G_n\in\PP$ with
$|V(G_n)|\to\infty$ such that $G_n\to W$ (i.e., $W_{G_n}$ converges to $W$ in
the $\delta_\square$ metric).
\end{definition}

Clearly $\overline{\PP}$ is closed under isomorphism, i.e. it is a
graphon property. The following examples show that $\overline{\PP}$
is not necessarily an ``extension'' of $\PP$ in the sense that $\PP$
cannot be recovered from it. Intuitively, $\overline{\PP}$ is a nice
object which is a ``clean'' version of $\PP$; it is an analytic
profile of the property $\PP$, which eliminates all uncontrollable
noise from it.

\begin{example}\label{TEST-GR-EX}
(a) Let $\PP$ be the graph property that the graph doesn't have a
4-cycle. Then only the $0$ function has property $\overline{\PP}$. In
fact, graphs without $4$-cycles are sparse and property testing (in
the sense of Definition \ref{PROPTEST}) does not distinguish sparse
graphs from each other.

\smallskip

(b)  Let $\PP$ be the graph property that the graph has an even
number of edges. Rather counter-intuitively, this property is
testable according to Definition \ref{PROPTEST} above, and its
closure is the whole set $\WW_0$.

\smallskip

(c) Let $\PP$ be the graph property that the graph has an even number
of nodes. This property is not testable, since adding a single node
to a large graph changes the distribution of small induced subgraphs
by very little, but changes the property. The closure of this
property is again the whole set $\WW_0$ (which is testable).

\smallskip

(d) Quasirandomness is defined as a property of a sequence of graphs,
but we can make it a graph property $\QQ$ (at the cost of a somewhat
arbitrary choice of the error bound) as follows: a graph $G$ on $n$
nodes is {\it quasirandom}, if
\[
\Bigl| t(K_2,G) - \frac12\Bigr| \le \frac1{\log n} \quad \text{and}
\quad t(C_4,G) \le \frac1{16} + \frac{1}{\log n}.
\]
The closure $\overline\QQ$ of this property consists of only one
function, the identically $1/2$ function. This singleton set of
functions is not testable, since for any sequence $(G_n)$ of
quasirandom graphs, $\|W_{G_n}-\frac12\|_\square\to 0$ but
$\|W_{G_n}-\frac12\|_1=1/2$. This implies (by Theorem \ref{char1}
below) that quasirandomness is not a testable property.
\end{example}

The first part of the following fact was stated in \cite{BCLSSV}.

\begin{prop}\label{lezaras}
{\rm (a)} The closure of a hereditary graph property $\PP$
consists of those functions $W\in\WW_0$ for which
\[
\Pr(\Ge(k,W)\notin\PP)=0
\]
for every $k$. Equivalently, $t_\ind(F,W)=0$ whenever
$F\notin\PP$.

\smallskip

{\rm (b)} The closure of a testable graph property $\PP$ consists of
those functions $W\in\WW_0$ for which
\[
\E(d_1(\Ge(k,W),\PP))\to 0\qquad (k\to\infty).
\]
\end{prop}

\noindent It follows from Theorem \ref{REG-TEST} that in the last
formula, we could replace $d_1$ by $d_\square$.

\begin{proof}
(a) Assume that $\Pr(\Ge(k,W)\notin\PP)=0$ for every $k$. Since
$\Ge(k,W)\to W$ with probability $1$ as $k\to\infty$, it follows that
$W\in\overline\PP$.

To show the converse, assume that $W\in\overline{\PP}$, and let
$(G_n)$ be a sequence of graphs such that $G_n\in\PP$ and $G_n\to W$.
We can write
\begin{equation}\label{PROB-P}
\Pr(\Ge(k,W)\notin\PP) = \sum_{F\in\FF_k\setminus\PP} \Pr(\Ge(k,W)=F)
= \sum_{F\in\FF_k\setminus\PP} t_\ind(F,W),
\end{equation}
so it suffices to prove that $t_\ind(F,W)=0$ for $F\notin\PP$. By
\eqref{T-IND}, we have
\begin{equation}\label{T-ERROR}
|t_\ind(F, G_n)-t_\ind(F,
W_{G_n})|<\frac{\binom{k}{2}}{n-\binom{k}{2}}
\end{equation}
for every $F\in\FF_k$. For $F\in\FF_k\setminus\PP$, we have
$t_\ind(F, G_n)=0$ since $G_n\in\PP$ and the property is hereditary.
Hence by \eqref{T-ERROR}, in this case
\[
t_\ind(F, W_{G_n}) \to 0\qquad (n\to\infty).
\]
But $G_n\to W$ implies that
\[
t_\ind(F, W_{G_n}) \to t_\ind(F, W),
\]
which proves (a).

\medskip

(b) The argument is similar, but a little more complicated. First,
assume that $\E(d_1(\Ge(k,W),\PP)\bigr)\to 0$ as $k\to\infty$. Then
we can select a sequence $k_1<k_2<\dots$ of integers such that
\[
\Pr\bigl(d_1(\Ge(k_m,W),\PP)>\frac1m\bigr) \le \frac1{2^m}.
\]
Since the right hand sides have a finite sum, it follows by the
Borel-Cantelli Lemma that $d_1(\Ge(k_m,W),\PP)\to 0$ with probability
$1$. Hence with probability $1$, there are graphs $H_m\in\PP$ such
that $d_1(\Ge(k_m,W),H_m)\to 0$, which implies that
$\delta_1(\Ge(k_m,W),H_m)\to 0$. Furthermore, if the sequence $(k_m)$
is sufficiently sparse, we have $\Ge(k_m,W)\to W$ with probability
$1$ as $k\to\infty$, and so $H_m\to W$ with probability $1$. Thus
$W\in\overline\PP$.

Second, assume that $W\in\overline{\PP}$, and let $(G_n)$ be a
sequence of graphs such that $G_n\in\PP$ and $G_n\to W$. Fix any
$\eps>0$. By the remark after the definition of property testing,
there is a graph property $\PP'$ and a natural number $k_\eps$ such
that for all $k,n$ with $k_\eps\le k\le |V(G_n)|$ we have
$\Pr\bigl(\Ge(k,G_n)\notin\PP'\bigr)\le \eps$, but
$\Pr\bigl(\Ge(k,G)\notin\PP'\bigr)> 1/2$ for all graphs $G$ with
$|V(G)|\ge k$ and $d_1(G,\PP)>\eps$. By \eqref{PROB-P} we have
\begin{align*}
\Pr\bigl(\Ge(k,W)\notin\PP'\bigr)&= \sum_{F\in\FF_k\setminus\PP'} t_\ind(F,W)\\
&\le\sum_{F\in\FF_k\setminus\PP'}
t_\ind(F,G_n)+\sum_{F\in\FF_k\setminus\PP'}
|t_\ind(F,G_n)-t_\ind(F,W_{G_n})|\\
&~~~+\sum_{F\in\FF_k\setminus\PP'} |t_\ind(F,W_{G_n})-t_\ind(F,W)|.
\end{align*}
Here the first sum is $\Pr\bigl(\Ge(k,G_n)\notin\PP'\bigr)\le \eps$.
The other two sums tend to $0$ as $n\to\infty$ by \eqref{T-ERROR} and
by $G_n\to W$. So for $k\ge k_0$, we get that
$\Pr\bigl(\Ge(k,W)\notin\PP'\bigr)\le \eps$.

Note that $\Ge(k,W)$ has the same distribution as $\Ge(k,\Ge(n,W))$
for $n\ge k$. Then for large enough $k,n$ we have
\[
\Pr\bigl(\Ge(k,W)\notin\PP'\bigr) \ge
\frac12\Pr\bigl(d_1(\Ge(n,W),\PP)>\eps\bigr),
\]
and so $\Pr\bigl(d_1(\Ge(n,W),\PP')>\eps\bigr)\le 2\eps$. Hence
\[
\E\bigl(d_1(\Ge(n,W),\PP)\bigr) \le (1-2\eps)\cdot\eps + (2\eps)\cdot
1 <3\eps.
\]
Since $\eps$ was arbitrary, this proves that
$\E\bigl(d_1(\Ge(n,W),\PP)\bigr) \to 0$.
\end{proof}

Part (a) of Proposition \ref{lezaras} implies:

\begin{corollary}\label{COR:HERED-SUPP}
If $\PP$ is a hereditary graph property, then $W\in\overline{\PP}$
depends only on the support of $W$.
\end{corollary}

We conclude with two lemmas on testability.  We'll say more about
both (Theorems \ref{COMBCHAR} and \ref{char1}), but these simple
lemmas will be needed before that.

\begin{lemma}\label{kozel2}
Let $\PP$ be a testable graph property. Then for every $\eps>0$ there
is an $\eps'>0$ and a positive integer $n'$ such that if  $G$ are two
graphs with $G'\in\PP$, $\delta_{\square}(G,G')<\eps'$ and
$|V(G)|,|V(G')|\geq n'$, then $d_1(G,\PP)<\eps$.
\end{lemma}

\begin{proof}
Let $\eps>0$ and let $k=k(\eps)$ in the definition of testability. We
show that $n'=9k^2$ and $\eps'=1/k^2$ is a good choice. Since $G'$
has property $\PP$, we have that at least $2/3$ of its $k$-node
induced subgraphs have property $\PP'$. It follows by \eqref{kozel}
that more than $1/3$ of the $k$-node subgraphs of $G$ have the
property $\PP'$. Hence $d_1(G,\PP)<\eps$ by testability.
\end{proof}

\begin{lemma}\label{P-BARP-TEST}
If $\PP$ is testable then $\overline{\PP}$ is testable.
\end{lemma}

The converse is not true in general, but in Theorem \ref{char1} we
will give a characterization of testable properties in terms of their
closure.

\begin{proof}
It suffices to prove that if $(W_n)$ is a sequence of functions in
$\WW_0$ such that $d_{\square}(W_n,\overline{\PP})\to 0$, then
$d_1(W_n,\overline\PP)\to 0$. We may assume that the sequence $W_n$
is convergent, so $W_n\to U$ for some $U\in\WW_0$ (in the
$\delta_\square$ distance). Clearly $U\in\overline\PP$, so by the
definition of closure, there are graphs $H_n\in\PP$ such that
$|V(H_n)|\to\infty$ and $H_n\to U$.

Fix any $\eps>0$. By Lemma \ref{kozel2}, there is an $\eps'>0$ such
that if $|V(G)|,|V(H)|$ are large enough, $H\in\PP$, and
$\delta_\square(G,H)<\eps'$, then $d_1(G,\PP)<\eps$. Furthermore,
there is an $n_\eps\ge 1$ such that if $n\ge n_\eps$, then
$\delta_\square(W_{H_n},U),
\delta_\square(W_n,U)\le \eps'/3$.

Fix any $n\ge n_\eps$, and let $G_{n,m}$ ($m=1,2,\dots$) be a
sequence of graphs such that $|V(G_{n,m})|\to\infty$ and $G_{n,m}\to
W_n$ as $m\to\infty$. Then
\[
\delta_\square(H_n,G_{n,m})\le \delta_\square(W_{H_n},U) +
\delta_\square(U,W_n) + \delta_\square(W_n, W_{G_{n,m}})<\eps'
\]
if $m$ is large enough, hence by the choice of $\eps'$, we have
$d_1(G_{n,m},\PP)\le\eps$. This means that there are graphs
$J_{n,m}\in\PP$ with $V(J_{n,m})=V(G_{n,m})$ such that
$d_1(G_{n,m},J_{n,m})\le \eps$. By choosing a subsequence, we can
assume that $J_{n,m}\to U_n$ as $m\to\infty$ for some
$U_n\in\overline{\PP}$. Applying Lemma \ref{fura} we obtain that
\[
d_1(W_n,\PP)\le \delta_1(W_n,U_n)\le \liminf_{m\to\infty}
\delta_1(W_{G_{n,m}},W_{J_{n,m}}) \le \liminf_{m\to\infty}
d_1(G_{n,m},J_{n,m})\le \eps.
\]
This proves the Lemma.
\end{proof}

\subsubsection{Closure and distance}

What is the relationship between the $d_1$ distance from a property
and from its closure? The following propositions summarize what we
know. We say that a graph $G'$ is an {\it equitable $m$-blowup} of
$G$ if it is obtained by replacing each node of $G$ by $m$ or $m+1$
twin copies ($m\ge 1$).

\begin{prop}\label{G2WG}
{\rm (a)}  For every hereditary graph property $\PP$ and every graph
$G$,
\[
d_1(G,\PP)\le d_1(W_G,\overline\PP).
\]
{\rm (b)} For every testable graph property $\PP$,
\[
|d_1(G,\PP)-d_1(W_G,\overline\PP)|\to 0 \qquad (|V(G)|\to\infty).
\]
{\rm (c)} Let $\PP$ be an arbitrary graph property and let $G^1,
G^2,\dots$ be all equitable blowups of a graph $G$. Then
\[
\liminf d_1(G^n,\PP) = d_1(W_G,\overline{\PP}).
\]
\end{prop}

\begin{proof}\killtext{GPR}
(a) Let $\delta>0$ and $U\in\overline\PP$ be such that
$\|W_G-U\|_1\le d_1(W_G,\overline\PP)+\delta$. Using Corollary
\ref{COR:HERED-SUPP}, we may assume that $U$ is a $0-1$-valued
function. From the fact that $\Ge(n,U)$ has property $\PP$ with
probability 1, it follows that $G'=\Ge'(n,U)$ has property $\PP$ with
probability 1. By Lemma \ref{LEM:GPRIMEL1}
\[
\E(d_1(G,G'))= \|W_G-U\|_1\le d_1(W_G,\overline\PP)+\delta.
\]
Hence there is an instance of $G'$ for which $G'\in\PP$ and
$d_1(G,G')\le d_1(W_G,\PP)+\delta$. Thus $d_1(G,\PP)\le d_1(G,G')\le
d_1(W_G,\overline\PP)+\delta$. Since $\delta$ is arbitrary, this
proves (a).

\smallskip

(b) Suppose not, then there exists a sequence of graphs $(G_n)$ with
$|V(G_n)|\to\infty$ such that $d_1(G_n,\PP)\to a$ and
$d_1(W_{G_n},\overline\PP)\to b$, where $a\not= b$. We may assume
that $V(G_n)=[q_n]$, where $q_n\to\infty$.

First, select graphs $H_n\in\PP$ such that $V(H_n)=V(G_n)$ and
$d_1(G_n,H_n)=d(G_n,\PP)$. By selecting a subsequence, we may assume
that the sequence $H_n$ is convergent; let $U\in\WW_0$ be its limit.
Clearly $U\in\overline\PP$. Then
$\delta_\square(W_{H_n},\overline\PP)\le
\delta_\square(W_{H_n},U)\to 0$, and hence $d_\square(W_{H_n},\overline\PP)\to
0$. By Theorem \ref{PROPTEST-FN}, this implies that
$d_1(W_{H_n},\overline\PP)\to 0$. But then in the inequality
\[
d_1(W_{G_n},\overline\PP) \le d_1(W_{G_n},W_{H_n})
+d_1(W_{H_n},\overline\PP)
\]
the first term on the right hand side tends to $a$, while the second
tends to $0$, showing that $b\le a$.

Second, fix $\eps>0$, and let $U_n\in\overline\PP$ be such that
$d_1(W_{G_n},U_n) \le d_1(W_{G_n},\overline{\PP})+ \eps$. By the
definition of testability of functions, there is an $m\ge 1$ such
that for $m\ge m_0$ and every $n$, we have that $\Ge(m,U_n)\in\PP'$
with probability at least $2/3$.

Similarly as in the proof of (a), consider the random graph
$G'=\Ge'(q_n,U_n)$. Clearly
\begin{equation}\label{GNGX}
d_1(G_n, \PP) \le d_1(G_n,G')+d_1(G', \PP).
\end{equation}
Here by Lemma \ref{LEM:GPRIMEL1}
\begin{equation}\label{EQ:EGNGX}
\E(d_1(G_n,G'))=d_1(W_{G_n},U_n).
\end{equation}
Next we show that, with probability tending to $1$ as $n\to\infty$,
we have
\begin{equation}\label{GNUN}
d_1(G_n,G')\le d_1(W_{G_n},U_n)+\eps.
\end{equation}
Indeed, the left hand side is the sum of $q_n/2$ independent random
variables, all $0-1$ valued, so this follows by the Law of Large
Numbers.

To estimate the other term in \eqref{GNGX}, let $Z$ be a random
$m$-element subset of $X$, and let $Y$ be a random $m$-elements
subset of $[0,1]$. The distributions of $Z$ and $Y$ are very close;
indeed, we can generate $Z$ by generating $Y$ and accepting it if its
elements fall into different intervals $[i/q_n, (i+1)/q_n]$,
$i=0,\dots,q_n-1$, and only regenerate otherwise. The probability
that we need to regenerate tends to $0$ as $n\to\infty$, so the total
variation distance of $Z$ and $Y$ tends to $0$. The probability that
$\Ge(Y,U_n)\in\PP'$ is at least $2/3$ as $\overline\PP$ is testable
by Lemma \ref{P-BARP-TEST}, and so the probability that
$\Ge(Z,U_n)\in\PP'$ is at least $1/2$. But $\Ge(Z,U_n)$ has the same
distribution as a random $m$-node subgraph of $G'$, and so by
testability,
\[
\Pr\bigl(d_1(G',\PP)\le\eps\bigr) + \frac13
\Pr\bigl(d_1(G',\PP)>\eps\bigr) \ge \frac12.
\]
This implies that with probability at least $1/4$,
\begin{equation}\label{GXPP}
d_1(G', \PP)\le\eps.
\end{equation}
With positive probability, both \eqref{GNUN} and \eqref{GXPP} occur,
and so by \eqref{GNGX} we have
\[
d_1(G_n, \PP) \le d_1(G_n,G')+d_1(G', \PP) \le d_1(W_{G_n},U_n)+\eps
+ \eps \le d_1(W_{G_n},\overline\PP)+3\eps.
\]
Sending $n\to\infty$ we see that $a\le b+3\eps$. Since $\eps$ was
arbitrary, it follows that $a\le b$, which is a contradiction.

(c) Let $\eps>0$, and let $U\in\overline{\PP}$ be a function such
that $\|W_G-U\|_1\le d_1(W_G,U)+\eps$. Let $H_n$ be a sequence of
graphs such that $H_n\to U$ and $H_n\in\PP$. Then for an appropriate
labeling of the nodes of $H_n$, we have $\|W_{H_n}-U\|_\square \to
0$. Since $W_G$ is $0-1$ valued, Lemma \ref{REV-CONV} implies that
$\|W_{H_n}-W_G\|_1 \to \|U-W_G\|_1$. Let $V(G)=\{1,\dots,k\}$ and
$V(H_n)=\{1,\dots,m\}$. Choose $n$ large enough so that
$\|W_{H_n}-W_G\|_1 \le\|U-W_G\|_1+\eps$ and $m\ge k/\eps$.

Let $V_i=\bigl\{\lfloor (i-1)m/k\rfloor +1,\dots, \lfloor
im/k\rfloor\bigr\}$ for $i=1,\dots,k$. Then $(V_1,\dots,V_k)$ is a
partition of $V(H_n)$ into $k$ almost equal classes. Define a graph
$G'$ on $\{1,\dots,k\}$ by connecting $u\in V_i$ to $v\in V_j$ if and
only if $ij\in E(G)$. Then $G'$ is an equitable blowup of $G$.
Furthermore, $W_{G'}$ and $W_G$ differ only on stripes of width less
than $1/m$ along the orders of the squares on which $W_G$ is
constant, so $\|W_{G'}-W_G\|_1 \le 2k/m\le 2\eps$. Thus
\begin{align*}
 d_1(G',\PP)&\le d_1(G',H_n) = \|W_{G'}-W_{H_n}\|_1 \le
\|W_{G'}-W_G\|_1+\|W_G-W_{H_n}\|_1\\
&\le 2\eps + \bigl(\|U-W_G\|_1+\eps\bigr) \le
d_1(W_G,\overline{\PP})+4\eps.
\end{align*}
Since $\eps$ was arbitrary, this proves that
\[
\liminf d_1(G^n,\PP) \le d_1(W_G,\overline{\PP}).
\]
To prove the converse, let $a=\liminf d_1(G^n,\PP)$, and let
$H_n\in\PP$ be chosen so that $V(H_n)=V(G^n)$ and $d_1(H_n, G^n) =
d_1(G^n,\PP)$. Select a subsequence such $d_1(H_n-G^n)\to a$, and
choose a further subsequence so that $H_n$ is convergent. Let $H_n\to
U\in\overline{\PP}$. We have $\delta_\square(W_{H_n},U)\to 0$ and
$\delta_\square(W_{G^n},W_G)\to 0$, hence by Lemma \ref{fura},
\[
\delta_1(W_G,U) \le \liminf\delta_1(W_{H_n},W_{G_n}) \le a.
\]
Hence $d_1(W_G,\overline{\PP})\le \delta_1(W_G,U)\le a$.
\end{proof}

\begin{prop}\label{prop:fura1}
Let $\PP$ be any graph property and $G_n\to W$, a convergent graph
sequence. Then
\[
\liminf_{n\to\infty} d_1(G_n,\PP) \ge d_1(W,\overline{\PP}),
\]
\end{prop}

\begin{proof}
Let $H_n\in\PP$ be such that $V(G_n)=V(H_n)$ and
$d_1(G_n,H_n)=d_1(G_n,\PP)$. We may select a subsequence so that
$H_n\to U$ for some $U\in\WW_0$. Clearly $U\in\overline{\PP}$.
Furthermore, $\|W_{G_n}-W\|_\square\to 0$ and
$\|W_{H_n}-U\|_\square\to 0$, so by Lemma \ref{fura1}, we have
\[
d_1(W,\PP) \le \|W-U\|_1 \le \liminf \|W_{G_n}-W_{H_n}\|_1= \liminf
d_1(G_n,\PP),
\]
a contradiction.
\end{proof}

\subsubsection{Monotone closure}

For two functions $U,W\in\WW_0$ we write $U\preceq W$ if there exist
functions $U',W'\in\WW$ such that $U\cong U'$, $W\cong W'$, and
$U'\le W'$ almost everywhere.

Let $\PP$ a graph property. By its {\it upward closure} we mean the
graph property $\PP^\uparrow$ consisting of those graphs that have a
spanning subgraph in $\PP$. For a function property $\RR$, we define
its {\it upward closure} to consist of those functions $W\in\WW_0$
for which there exists a function $U\in\RR$ such that $U\preceq W$.
In both versions, the {\it downward closure} is defined analogously.

The following theorem, whose proof is surprisingly nontrivial,
asserts that closure and upward closure commute.

\begin{theorem}\label{THM:UP}
For every graph property $\PP$,
\[
\overline{\PP^\uparrow} =(\overline{\PP})^\uparrow.
\]
\end{theorem}

\begin{proof}
First, let $W\in\overline{\PP^\uparrow}$. Then there exists a graph
sequence $G_n\to W\in\WW_0$ such that $G_n$ has a spanning subgraph
$G_n'\in\PP$. We consider the pair $(G_n,G_n')$ as the graph $G_n$ in
which the edges of $G_n'$ colored red, the remaining edges are
colored blue. We may choose a subsequence of the indices so that the
remaining sequence is convergent as {\it 2-edge-colored graphs},
meaning that for every 2-edge-colored simple graph $F$, the sequence
of densities $t(F,G_n)$ of the densities of color-preserving
homomorphisms is convergent. It is shown in \cite{LSz3} that the
limit object of a such a sequence can be described by a pair of
functions $U,V\in\WW_0$, such that $V$ is the limit of the sequence
$(G_n)$, $U$ is the limit of the sequence $(G_n')$, and $U\le V$
almost everywhere. Hence $V\cong W$ and $U\in\overline{\PP}$. This
proves that $W\in (\overline{\PP})^\uparrow$.

Conversely, let $W\in (\overline{\PP})^\uparrow$, then there is a
$U\in\overline{\PP}$ and a $V\cong W$ such that $U\le V$. Let
$G_n\in\PP$ be such that $G_n\to U$. By Lemma 4.16 in \cite{BCLSV1},
we may label the graphs $G_n$ so that $\|W_{G_n}-U\|_\square\to 0$.
Let $\LL=\LL_n$ denote the partition of $[0,1]$ into $N=|V(G_n)|$
equal intervals $\{I_1,\dots,I_N\}$, and for $W\in\WW$, let $W_\LL$
denote the function obtained by replacing $W$ by its average on each
of the intervals $I_i\times I_j$. Clearly $U_\LL \le V_\LL$, and
\begin{equation}\label{EQ:WWL}
\|U-U_\LL\|_\square \to 0, \qquad \|V-V_\LL\|_\square \to 0
\end{equation}
(this holds even with the $L_1$-norm in place of the cut norm), and
so
\begin{equation}\label{EQ:ULWGN}
\|U_\LL-W_{G_n}\|_\square\to 0.
\end{equation}
Define
\[
W'_n=V_\LL+\frac{1-V_\LL}{1-U_\LL}(W_{G_n}-U_\LL)
\]
(where the second term is $0$ wherever $U_\LL=1$). Then
\[
W'_n\le V_\LL+\frac{1-V_\LL}{1-U_\LL}(1-U_\LL)=1,
\]
and
\[
W'_n= V_\LL\Bigl(1-\frac{W_{G_n}-U_\LL}{1-U_\LL}\Bigr) +
\frac{W_{G_n}-U_\LL}{1-U_\LL} \ge
U_\LL\Bigl(1-\frac{W_{G_n}-U_\LL}{1-U_\LL}\Bigr)+
\frac{W_{G_n}-U_\LL}{1-U_\LL} = W_{G_n}.
\]
By Lemma \ref{WSTAR},
\begin{equation}\label{EQ:WNVL}
\|W'_n-V_\LL\|_\square =
\Bigl\|\frac{1-V_\LL}{1-U_\LL}(W_{G_n}-U_\LL)\Bigr\|_\square \to 0
\qquad (n\to\infty).
\end{equation}
Since $W'_n$ is a stepfunction that is constant on intervals $I\times
J$ $(I,J\in\LL)$, it can be viewed as $W_{H_n}$ for some weighted
graph $H_n$ on $[N]$. Create a random graph $G_n'$ as follows: for
$1\le i < j\le N$, connect $i$ and $j$ with probability equal to
weight of the edge in $H_n$. Lemma 4.3 in \cite{BCLSV1} implies that
with probability at least $1-e^{-N}$,
\[
\|W_{G'_n}-W_{H_n}\|_\square \le \frac{10}{\sqrt{N}},
\]
and hence
\begin{equation}\label{EQ:HNGN}
\|W_{G'_n}-W'_n\|_\square \to 0 \qquad (n\to\infty).
\end{equation}
Trivially, $G_n$ is a subgraph of $G_n'$ with probability $1$, and
\eqref{EQ:WWL}, \eqref{EQ:ULWGN} and \eqref{EQ:HNGN} imply that
\[
\|W_{G'_n}-V\|_\square \le \|W_{G'_n}-W'_n\|_\square +
\|W'_n-V_\LL\|_\square + \|V_\LL-V\|_\square \to 0.
\]
This proves that $V\in\overline{\PP^\uparrow}$, and so
$W\in\overline{\PP^\uparrow}$.
\end{proof}

\subsubsection{Robustness}

The following definition will be important in our characterization of
testable graph properties.

\begin{definition}
A graph property $\PP$ is {\it robust}, if for every $\eps>0$ there
are numbers $n=n(\eps)>0$ and $\eps'>0$ such that if $G$ is a graph
with $|V(G)|\geq n(\eps)$ and $d_1(W_G,\overline{\PP})\leq \eps'$,
then $d_1(G,\PP)\leq\eps$.
\end{definition}

Proposition \ref{G2WG}(b) implies that every testable property, in
particular every hereditary property, is robust. The following fact,
which is an immediate consequence of Proposition \ref{G2WG}, provides
a combinatorial criterion for robustness.

\begin{prop}\label{BLOWUP-ROB}
A graph property $\PP$ is robust if and only if for every $\eps>0$
there is an $\eps'>0$ and an $n_\eps\ge 1$ such that if $G$ is a
graph with $|V(G)|\ge n_\eps$ and $G$ has infinitely many equitable
blowups $G'$ with $d_1(G',\PP)\le\eps'$, then $d_1(G,\PP)\le\eps$.
\end{prop}

\subsubsection{Characterizing testability of graph properties}

Our next main theorem shows the relationship between analytic and
graph theoretic testability.

\begin{theorem}\label{char1}
A graph property $\PP$ is testable if and only if it is robust and
its closure $\overline{\PP}$ is testable.
\end{theorem}

\begin{proof}
We have seen that every testable graph property is robust
(Proposition \ref{G2WG}(b)), and that its closure is testable (Lemma
\ref{P-BARP-TEST}). So to complete the proof, it suffices to prove
that if $\PP$ is robust and $\overline{\PP}$ is testable then $\PP$
is testable.

Let $\eps>0$. By the robustness of $\PP$, there is an $\eps'>0$ and
$k_\eps\ge 1$ such that if $|V(G)|\ge k_\eps$ and
$d_1(W_G,\overline{\PP})<\eps'$, than $d_1(G,\PP)<\eps$. By the
testability of $\overline\PP$ and Theorem \ref{PROPTEST-FN}, there is
an $\eps''>0$ such that if $d_\square(W, \overline{\PP})< \eps''$
then $d_1(W, \overline{\PP})<\eps'$. By the definition of $\bar\PP$
and by Theorem 2.9 in \cite{BCLSV1}, there exists an $n_\eps\ge
k_\eps$ such that

(i) for every graph $G\in\PP$ with $|V(G)|\geq n_\eps$, we have
$d_{\square}(W_G,\overline{\PP})<\eps''/4$;

(ii) for every graph $G$ with $|V(G)|\geq n_\eps$ and every $n_\eps
\le m\le |V(G)|$, we have $\delta_\square(\Ge(m,G),G)<\eps''/4$ with
probability at least $2/3$.

Let $\PP'$ denote the property of a graph $G$ that
$d_{\square}(W_G,\overline{\PP})\le \eps''/2$ (this depends on
$\eps$, but as we remarked after the definition, this is OK). We
claim that $\PP'$ is a good test property for $\PP$ (for the given
$\eps$).

Let $G$ be a graph, $n_\eps\le n\le |V(G)|$, and let $H=\Ge(n,G)$.
First, suppose that $G\in\PP$. By (i),
$d_{\square}(W_G,\overline{\PP})<\eps''/4$. Further by (ii),
$\delta_\square(H,G)=\delta_\square(W_H,W_G)\le \eps''/4$ with
probability at least $2/3$. In every such case,
\[
d_{\square}(W_H,\overline{\PP})\le \delta_\square(W_H,W_G) +
d_{\square}(W_G,\overline{\PP})< \eps''/2.
\]
Thus $H$ has property $\PP'$.

Second, suppose that $d_1(G,\PP)\geq \eps$; we want to prove by
contradiction that $H$ does not have the property $\PP'$ with
probability at least $2/3$. Assume that it is not true, then with
probability larger than $1/3$, $d_\square(W_H,\overline\PP)\le
\eps''/2$. We have $d_{\square}(W_G,W_H)\le \eps''/4$ with
probability more than $2/3$. This implies that there exists at least
one induced subgraph $H$ of $G$ with $n$ nodes such that
$d_{\square}(W_G,W_H)<\eps''/4$ and
$d_{\square}(W_H,\overline{\PP})\leq \eps''/2$. We obtain that
$d_{\square}(W_G,\overline{\PP})<\eps''$. It follows from the
testability of $\overline{\PP}$ that $d_1(W_G,\overline{\PP})<\eps'$.
It follows from the robustness of $\PP$ that $d_1(G,\PP)< \eps$, a
contradiction.
\end{proof}

A further connection between testable graph properties and graphon
properties is the following fact:

\begin{theorem}\label{genprop}
For every closed testable graphon property $\RR$ there exists a
testable graph property $\PP$ such that $\RR=\overline{\PP}$.
\end{theorem}

\begin{proof}
Let $\eps>0$. Since $\RR$ is testable, it follows by Theorem
\ref{PROPTEST-FN} that there is and $\eps'>0$ such that
$B_{\square}(\RR,\eps') \subseteq B_1(\RR,\eps)$. The set $\RR$, with
the distance function $\delta_{\square}$, is a compact metric space,
and hence it can be covered by a finite number of balls
$B_i=B_\square(W_i,\eps'/2)$, $i=1,\dots,m$. For every $1\le i\le m$,
there is a smallest positive integer $n_i$ such that if $n\geq n_i$
then there is a graph $G$ with $n$ nodes such that $W_G\in B_i$. Let
$n_{\eps}=\max_i n_i$. So for every $U\in\RR$, and $n\ge n_\eps$
there is a graph $G$ with $n$ nodes such that (using the $i$ for
which $B_i\ni U)$,
\[
\delta_\square(W_G,U)\le \delta_\square(W_G,W_i) + \delta_\square(W_i,U)
\le \frac{\eps'}{2}+\frac{\eps'}{2}=\eps'\le\eps.
\]
This also implies that $\delta_1(W_G,\RR)\le \eps$.

For every $n\ge 1$, let
\[
\eps_n=\max\bigl(\frac{50}{\sqrt{\log\log n}},~\sup\{\eps:~n_\eps \ge
n\}\bigr).
\]
Clearly $\eps_n\searrow 0$ as $n\to\infty$, and $\eps_{n_\eps}\ge
\eps$. We prove that the property
$\PP=\{G:~\delta_1(W_G,\RR)\leq\eps_{|V(G)|}\}$ is robust and its
closure is $\RR$.

First we show that $\RR\supseteq\overline{\PP}$. Indeed, if
$W\in\overline{\PP}$, then there is a sequence $G_n\in\PP$ with
$|V(G_n)|\to\infty$ such that $\delta_\square(W_{G_n},W)\to 0$. Since
$G_n\in\PP$ means that $\delta_\square(W_{G_n},\RR)\le
\delta_1(W_{G_n},\RR)\le \eps_{|V(G_n)|}\to 0$, it follows that
$\delta_\square(W,\RR)= 0$. Since $\RR$ is closed, this implies that
$W\in\RR$.

To show that $\RR\subseteq\overline{\PP}$, consider any $U\in\RR$,
and let $\eps>0$. As we have seen above, there is a graph $G\in\PP$
with $n_\eps$ nodes such that $\delta_\square(W_G,U)\le\eps$ and
$\delta_1(W_G,\RR)\le\eps\le \eps_{n_\eps}$. So $G\in\PP$, which
shows that there is a function $W_G$ with $G\in\PP$ arbitrarily close
to $U$.

To show that property $\PP$ is robust, consider any graph $G$ with
$n$ nodes. Choose $U\in\RR$ so that $\|W_G-U\|_1 \le 2d_1(W_G,\RR)$.
By Corollary \ref{COR:SQUARECLOSE}, there is a graph $\widetilde{G}$
on $n$ nodes such that $d_1(\widetilde{G},U)\le 4\|W_G-U\|_1 \le
8d_1(W_G,\RR)$ and $\|W_{\widetilde{G}}-U\|_\square
\le\frac{50}{\sqrt{\log\log n}}\le \eps_n$. Thus
$\widetilde{G}\in\PP$, and
\begin{align*}
d_1(G,\PP) &\le d_1(G,\widetilde{G})= \|W_G-W_{\widetilde{G}}\|_1 \le
\|W_G-U\|_1 +\|\widetilde{G}-U\|_1\\ &\le 2d_1(W_G,\RR) +
8d_1(W_G,\RR) = 10d_1(W_G,\RR).
\end{align*}

Thus $\PP$ is robust and $\overline{\PP}=\RR$ is testable. Theorem
\ref{char1} implies that $\PP$ is testable.
\end{proof}

Let us say that two graph properties $\PP$ and $\PP'$ are {\it
equivalent} if their closure is the same. Theorem \ref{genprop}
implies that equivalence classes of testable graph properties are in
a one to one correspondence with the testable graphon properties.

As an application of these results, we give a purely combinatorial
characterization of testability, which generalizes the result of Alon
and Shapira on the testability of hereditary properties, and also
contains a finite analogue of Theorem \ref{REG-TEST}.

\begin{theorem}\label{COMBCHAR}
For a graph property $\PP$, the following are equivalent:

\smallskip

{\rm (a)} $\PP$ is testable;

\smallskip

{\rm (b)} For every $\eps>0$ there is an $\eps'>0$ and an $n'>0$ such
that if $G\in\PP$, $G'$ is any other graph such that
$\delta_\square(G,G')<\eps'$, and $|V(G)|, |V(G')|\ge n'$, then
$d_1(G',\PP)<\eps$.

\smallskip

{\rm (c)} For every $\eps>0$ there is an $\eps_0>0$ and an $n_0>0$
such that if $G\in\PP$ and $G'$ is an induced subgraph of $G$ such
that $\delta_\square(G,G')<\eps_0$ and $|V(G')|\ge n_0$, then
$d_1(G',\PP)<\eps$.
\end{theorem}

\begin{remark}
Condition (b) says that, roughly speaking, if a graph $G$ is close to
a graph $H\in\PP$ in the $\delta_\square$ distance, then it is also
close to a (possibly different) graph $J\in\PP$ in edit distance. But
we need to be careful: Let $\PP$ be the (trivial) graph property of
having at most $1$ node, then a large edgeless graph will be close to
$\PP$ in the $\delta_\square$ distance, but not in $d_1$. The theorem
shows that it is enough to add the assumption that the graphs are
large enough.

Condition (c) is a weakening of Alon--Shapira condition that the
graph is hereditary. Theorem 2.11 in \cite{BCLSV1} implies that a
randomly chosen $k$-node induced subgraph of $G$ is closer to $G$
than $10/\sqrt{\log k}$ in the $\delta_\square$ distance, with large
probability. So we can think of induced subgraphs $G'$ satisfying
$\delta_\square(G,G')<\eps_0$ as ``typical''.
\end{remark}

\begin{proof}
Lemma \ref{kozel2} says that (a)$\Rightarrow$(b), and
(b)$\Rightarrow$(c) is trivial. To prove that (c)$\Rightarrow$(a), we
start with proving a version of the condition in the theorem for
functions.

\begin{claim}\label{FN-SAMPLE}
For every $\eps>0$ there is an $\eps_1>0$ and an $n_1>0$ such that if
$W\in\overline{\PP}$ and $G$ is graph such that $|V(G)|\ge n_1$,
$t_\ind(G,W)>0$ and $d_\square(W_G,W)<\eps_1$, then
$d_1(G,\PP)<\eps$.
\end{claim}

Given $\eps>0$, let $\eps_0$ and $n_0$ be as in the condition of the
theorem, and set $\eps_1=\eps_0/2$, $n_1=n_0$. Let $H_n\in\PP$ be a
sequence of graphs such that $H_n\to W$. Then $t_\ind(G,H_n)\to
t_\ind(G,W)>0$, so for large enough $n$, $G$ is an induced subgraph
of $H_n$. Furthermore, we have $\delta_\square(G,H_n)\le
\delta_\square(W_G,W)
+\delta_\square(W,W_{H_n})\to\delta_\square(W_G,W)\le\eps_0/2$, and
so for large enough $n$ we have $\delta_\square(G,H_n)<\eps_0$. So
the condition of the theorem implies that $d_1(G,\PP)<\eps$. This
proves Claim \ref{FN-SAMPLE}.

To prove that $\PP$ is testable, we use Theorem \ref{char1}: it
suffices to prove that $\overline{\PP}$ is testable and $\PP$ is
robust. To prove that $\overline{\PP}$ is testable, we use Theorem
\ref{PROPTEST-FN}: We want to prove that if a function is close to
$\overline{\PP}$ in the $\|.\|_\square$ norm, then it is also close
in the $\|.\|_1$ norm. Our next step is proving a special case of
this.

\begin{claim}\label{W-FLEX}
For every $\eps>0$ there is an $\eps_2>0$ such that if
$W\in\overline{\PP}$ and $U\in\WW$ is a function such that $U=W$
wherever $W(x,y)\in\{0,1\}$ and $d_\square(W,U)<\eps_2$, then
$d_1(U,\overline{\PP})<\eps$.
\end{claim}

Let $\eps_1$ and $n_1$ be as in Claim \ref{FN-SAMPLE}, and set
$\eps_2=\eps_1/2$. Choose a sequence $G_n$ of graphs such that
$G_n\to U$ and $t_\ind(G_n,U)>0$. (For example, random graphs
$G_n=\Ge(n,U)$ have this property with large probability.) By the
condition on $U$, we have $t_\ind(G_n,W)>0$. Furthermore,
$\delta_\square(W_{G_n},W)\le
\delta_\square(W_{G_n},U)+
\delta_\square(U,W) \to \delta_\square(U,W) < \eps_1/2$, and hence
for large enough $n$, we have $\delta_\square(W_{G_n},W)<\eps_1$. If
we also choose $n$ large enough so that $|V(G_n)|\ge n_1$, then
$d_1(G_n,\PP)<\eps$ by Claim \ref{FN-SAMPLE}. Thus there exist graphs
$H_n\in\PP$ with $V(H_n)=V(G_n)$ such that $d_1(G_n,H_n)\le \eps$. We
choose a subsequence such that $H_n\to W'$ for some function $W'$,
which is then in $\overline{\PP}$. Now Lemma \ref{fura} implies that
$\delta_1(U,W')< \eps$, and hence $d_1(U, \overline{\PP})<\eps$.

\begin{claim}\label{PPTEST}
$\overline{\PP}$ is testable.
\end{claim}

We use Theorem \ref{PROPTEST-FN}. Let $(U_n)$ be a sequence of
functions in $\WW_0$ such that $d_\square(U_n,\overline{\PP})\to 0$,
we want to prove that $d_1(U_n,\overline{\PP})\to 0$. We may assume
that $(U_n)$ is convergent in the $\delta_1$ distance, and hence
applying appropriate measure preserving transformations, we may
assume that there is a function $W\in\WW_0$ such that
$\|U_n-W\|_\square\to 0$. Let
\[
U_n'(x,y)=
  \begin{cases}
    W(x,y), & \text{if $W(x,y)\in\{0,1\}$}, \\
    U_n(x,y), & \text{otherwise}.
  \end{cases}
\]
Let $\eps>0$ and choose $n_2,\eps_2$ as in Claim \ref{W-FLEX}, with
input $\eps/2$. Just as in the proof of Theorem \ref{FLEX-TEST}, we
see that $\|U_n-U_n'\|_1\to 0$, so we can choose a large enough $n$
such that $\|U_n-U_n'\|_1<\min\{\eps/2,\eps_2/2\}$ and also
$\|U_n-W\|_\square<\eps_2/2$. Then $U_n'=W$ wherever $W\in\{0,1\}$
and $\|U'_n-W\|_\square\le \|U'_n-U_n\|_\square+\|U_n-W\|_\square\le
\|U'_n-U_n\|_1+\|U_n-W\|_\square<\eps_2$, and so by Claim
\ref{W-FLEX} we have $\|U'_n-W\|_1<\eps/2$. Thus $\|U_n-W\|_1\le
\|U'_n-U_n\|_1+\|U_n-W\|_1\le\eps$.

\begin{claim}\label{PPROBUST}
$\PP$ is robust.
\end{claim}

Let $\eps>0$, let $\eps_1$ and $n_1$ be chosen so that they satisfy
Claim \ref{FN-SAMPLE} with input $\eps/2$, and let
$\eps'=\min(\eps/4,\eps_1/3)$. Let $G$ be a graph with $|V(G)|=m\geq
n_1$ and $d_1(W_G,\overline{\PP})< \eps'$. Let $W\in\overline{\PP}$
be a function such that $\|W_G-W\|_1<\eps'$. Consider the random
graph $G'=\Ge'(n,W)$. Then $t_\ind(G',W)>0$ with probability $1$, and
by Lemma \ref{LEM:GPRIME}, with probability tending to $1$,
$\|W_{G'}-W\|_\square\le \eps'$ if $n$ is large enough. Furthermore,
Lemma \ref{LEM:GPRIMEL1} implies that
$\E(d_1(G,G'))=\|W_G-W\|_1<\eps'$, and so with probability at least
$1/2$, we have $d_1(G,G')<2\eps'<\eps/2$. In such a case
\[
\|W_G-W\|_\square \le \|W_G-W_{G'}\|_\square + \|W_{G'}-W\|_\square
\le\|W_G-W_{G'}\|_1 + \|W_{G'}-W\|_\square <3\eps'=\eps_1,
\]
and so by Claim \ref{FN-SAMPLE}, we have $d_1(G',\PP)<\eps/2$. Hence
$d_1(G,\PP)\le d_1(G,G')+d_1(G',\PP) <\eps$, which proves that $\PP$
is robust.

Using Theorem \ref{char1}, this completes the proof of Theorem
\ref{COMBCHAR}.
\end{proof}

\subsection{Property testing vs. parameter testing}

Parameter testing is a problem related to property testing, and in
some respects simpler; but the main facts are often analogous. It was
introduced in \cite{BCLSV1}, where a number of different
characterizations were also given (see also \cite{BCLSSV}). We
summarize the main results about parameter testing, to point out this
analogy; finally, we prove a direct connection between these notions.

A {\it graph parameter} is a function defined on isomorphism types
of graphs. A graph parameter $f$ is {\it testable}, if for every
$\eps>0$ there is a positive integer $k=k(\eps)$ such that if $G$
is a graph with at least $k$ nodes and $S$ is a random subset of
$k$ nodes of $G$ (chosen uniformly over all $k$-element sets),
then
\[
\Pr(|f(G)-f(G[S])|>\eps)<1/3.
\]
Testability of parameters is related to the convergence of graph
sequences:

\begin{prop}[\cite{BCLSV1}]\label{thm:TESTABLE-PAR}
A graph parameter $f$ is testable if and only if $f(G_n)$
converges for every convergent graph sequence $(G_n)$.
\end{prop}

A {\it graphon functional} is a real valued function defined on
graphons; equivalently, a functional defined on $\WW_0$ that is
invariant under isomorphism. A graphon functional $f$ is {\it
testable}, if there is a graph parameter $g$ such that for every
$\eps>0$ there is a positive integer $k(\eps)$ such that for every
$W\in\WW_0$ and $k\ge k(\eps)$,
\[
\Pr(|f(W)-g(\Ge(k,W))|>\eps)<1/3.
\]
Informally, we can estimate the value of $f(W)$ by generating a
$W$-random graph with sufficiently many nodes, and evaluating the
graph parameter $g$ on this.

Testable graphon parameters have a simple characterization:

\begin{prop}\label{FUNCT-TEST}
A graphon parameter is testable if and only if it is continuous in
the norm $\|.\|_\square$.
\end{prop}

\begin{proof}
Suppose that $f$ is testable. If $\|U-W\|_\square$ is small enough,
then the distributions of $\Ge(k,W)$ and $\Ge(k,U)$ are close, and so
there is a graph $F$ such that $|f(W)-g(F)|\le\eps$ and
$|f(U)-g(F)|\le\eps$, so $|f(W)-f(U)|\le 2\eps$. So $f$ is
continuous.

Conversely, if $f$ is continuous, then it is uniformly continuous by
the compactness of $\WW$, and so for every $\eps>0$ there is an
$\eps'>0$ such that if $\delta_\square(U,W)\le\eps'$ then
$|f(U)-f(W)|\le\eps'$. Define $g(F)=f(W_F)$ for a graph $F$. By
Theorem \ref{THM:SAMPLE2}, if $k$ is large enough, then with large
probability,
\[
\delta_\square(W_{\Ge(k,W)},W)\le \eps',
\]
and hence with large probability,
\[
|g(\Ge(k,W))-f(W)| = |f(W_{\Ge(k,W)})-f(W)|\le \eps.
\]
This proves that $f$ is testable.
\end{proof}

Theorem 5.1 in \cite{BCLSV1} implies the following.

\begin{prop}\label{PARAM-TEST}
A graph parameter is testable if and only if there is a testable
graphon parameter $f$ such that $f(W_G)-g(G)\to 0$ if $|V(G)|\to
\infty$.
\end{prop}

We can view this last fact as follows: Every graphon functional $f$
gives rise to a graph parameter $\widehat{f}$ by
$\widehat{f}(G)=f(W_G)$. This graph parameter $\widehat{f}$ is
testable if $f$ is testable, and testable graph parameters are
exactly those that are ``asymptotically equal'' to $\widehat{f}$ for
some testable graphon functional $f$.

Alon and Shapira \cite{AS} proved that the edit distance from every
hereditary graph property is a testable parameter. More generally,
Fischer and Newman \cite{FN} proved:

\begin{theorem}[Fischer--Newman]\label{dist-test}
A graph property $\PP$ is testable if and only if the edit distance
$d_1(G,\PP)$ is a testable parameter.
\end{theorem}

The following theorem gives an analogue of this result for graphons,
from which Theorem \ref{dist-test} can be deduced:

\begin{theorem}\label{PROPTEST-FN-DIST}
A graphon property $\RR$ is testable if and only if the distance
$d_1(.,\RR)$ is a testable functional.
\end{theorem}

The content of the theorem is that the $d_1$ distance from a testable
property is continuous function in the $\|.\|_\square$ norm. It is
trivial that this distance is continuous in the $\|.\|_1$ norm for
any graphon property. The second half of the proof below shows that
the $d_1$ distance from {\it any} graphon property is lower
semi-continuous in the $\|.\|_\square$ norm.

\begin{proof}
It is easy to see that if the functional $d_1(.,\RR)$ is testable
then $\RR$ is testable: the set $\{W:~d_1(W,\RR)<\eps\}$ is open in
the $\|.\|_\square$ norm, and contains the compact set
$\overline\RR$, so it contains a neighborhood $B_\square(\RR,\eps')$
of $\RR$ for some $\eps'>0$.

Now suppose that $\RR$ is testable. Let $W\in\WW_0$ and let $W_n\to
W$. We claim that $d_1(W_n,\RR)\to d_1(W,\RR)$. We may assume that
$\|W_n-W\|_\square\to 0$.

Let $\eps>0$, and let $U\in\RR$ be such that $\|W-U\|_1\le
d_1(W,\RR)+\eps$. By Lemma \ref{REV-CONV-1}, there is a sequence of
functions $U_n\in\WW$ such that $\|U_n-U\|_\square\to 0$ and
$\|U_n-W_n\|_1\to \|U-W\|_1$. By the testability of $\RR$ and by
Theorem \ref{PROPTEST-FN}, it follows that $\|U_n-U\|_1\to 0$, and so
\[
\|W_n-U\|_1 \le  \|W_n-U_n\|_1 + \|U_n-U\|_1 \to\|U-W\|_1.
\]
Hence
\[
\limsup_{n\to\infty} d_1(W_n,\RR) \le\limsup_{n\to\infty} \|W_n-U\|_1
\le  \|U-W\|_1 \le d_1(W,\RR)+\eps.
\]
Since $\eps>0$ is arbitrary, this implies that
\begin{equation}\label{EQ:UWLIMSUP}
\limsup_{n\to\infty} d_1(W_n,\RR) \le d_1(W,\RR)
\end{equation}

To prove the reverse, let $V_n\in\RR$ be chosen so that
$\|W_n-V_n\|_1 \le d_1(W_n,\RR)+1/n$. By selecting a subsequence, we
may assume that the sequence $(V_n)$ is convergent in the
$\delta_\square$ distance. Let $V\in\WW_0$ be its limit. Clearly
$V\in\RR$. Thus by Lemma \ref{fura} we have
\begin{equation}\label{EQ:UWLIMINF}
d_1(W,\RR) \le \delta_1(W,V) \le\liminf_{n\to\infty}
\delta_1(W_n,V_n) \le \liminf_{n\to\infty} d_1(W_n,\RR).
\end{equation}

The relations \eqref{EQ:UWLIMSUP} and \eqref{EQ:UWLIMINF} prove that
$d_1(W_n,\RR)\to d_1(W,\RR)$, and so $d_1(.,\RR)$ is continuous.
\end{proof}

The theorem of Fischer and Newman (Theorem \ref{dist-test}) is easy
to derive from here. By the results of \cite{BCLSV1} Theorem 6.1(d),
it suffices to prove that for every testable property $\PP$,
$d_1(W,\overline\PP)$ is a continuous function of $W$ in the
$\|.\|_\square$ norm, and $d_1(G,\PP)-d_1(W_G,\overline\PP)\to 0$ if
$|V(G)|\to\infty$. The first assertion follows from Theorem
\ref{PROPTEST-FN-DIST}, the second, from Proposition \ref{G2WG}(b).

\subsection{Flexible properties}\label{FLEXIBLE}

Let $W\in\WW_0$. A function $U\in\WW_0$ is a {\it flexing} of $W$ if
$U(x,y)=W(x,y)$ for all $x,y$ with $W(x,y)\in\{0,1\}$ (so we may
change the values of $W$ that are strictly between $0$ and $1$; note
that we may change these to $0$ or $1$, so the relation is not
symmetric). We say that a function property is {\it flexible} if it
is preserved under flexing. The following Proposition gives some
sufficient conditions for flexibility; the proofs are straightforward
and omitted.

\begin{prop}\label{PROP:FUNC-FLEX}
{\rm (a)} Every function property that implies that the function is
$0-1$ valued is flexible.

\smallskip

{\rm (b)} For every graph $F$, the function property
$\{W\in\WW_0:~t(F,W)=0\}$ is flexible.

\smallskip

{\rm (c)} If $\RR$ is a flexible function property, then the function
properties $\RR^*=\{1-W:~W\in\RR\}$, $\RR^\uparrow$ and
$\RR^\downarrow$ are flexible.

\smallskip

{\rm (d)} The intersection and union of any set of flexible
properties is flexible.
\end{prop}

We call a graph property {\it flexible} if its closure is flexible.

\begin{example}\label{ex:FLEX-GR}
(a) The graph property that $G$ is clique of size $\lceil
|V(G)|/2\rceil$, together with isolated nodes, is flexible, since its
closure consists of a single graphon (represented by the function $W$
that is $1$ if $x,y\le 1/2$ and $0$ otherwise).

\smallskip

(b) The graph property that $\omega(G)\ge |V(G)|/2$ is flexible
(where $\omega(G)$ is the size of the largest clique in $G$), since
it is the upward closure of the property in (a). Similarly, the
property that $\alpha(G)\ge |V(G)|/2$ is flexible (where $\alpha(G)$
is the size of the largest independent set in $G$).

\smallskip

(c) The graph property that there is a labeling of the nodes by
$\{1,\dots,n\}$ such that two nodes are connected if and only if
their labels sum to at most $n$, is flexible. Indeed, the closure of
this graph property consists of a single graphon (represented by the
function $W$ that is $1$ if $x+y\le 1$ and $0$ otherwise). This
closure is flexible by Proposition \ref{PROP:FUNC-FLEX}(a).

\smallskip

(d) The graph property that there is a labeling of the nodes by
$\{1,\dots,n\}$ such that all pairs whose labels sum to at most $n$
are connected by an edge, is flexible. Indeed, this is the upward
closure of the property in (c).
\end{example}

These and other examples follow from the following proposition:

\begin{prop}\label{PROP:GRAPH-FLEX}
{\rm (a)} Every graph property $\PP$ for which $\overline{\PP}$
consists of $0-1$ valued functions is flexible.

\smallskip

{\rm (b)} If a graph property is flexible, then so are its upward and
downward closures.

\smallskip

{\rm (c)} If a graph property $\PP$ is flexible, then so is the
property obtained by complementing all graphs in $\PP$.

\smallskip

{\rm (d)} Every hereditary graph property is flexible.
\end{prop}

\begin{proof}
(a) is obvious. (b) follows from Theorem \ref{THM:UP} and Proposition
\ref{PROP:FUNC-FLEX} (c). Assertion (c) follows from Proposition
\ref{PROP:FUNC-FLEX} (c). Finally, (d) follows from \ref{lezaras}
(a), since if $t(F,W)=0$, then $t(F,U)=0$ for every flexing $U$ of
$W$.
\end{proof}

Our main result about flexible properties is the following.

\begin{theorem}\label{FLEX-TEST}
{\rm (a)} Every closed flexible graphon property is testable.

\smallskip

{\rm (b)} Every robust flexible graph property is testable.
\end{theorem}

\begin{remark}
1. It is important to assume that the property is closed. We have
seen that the function property of being 0-1 valued is flexible and
it is trivially a graphon property, but it is not testable.

2. We can weaken this notion slightly in a way that we preserve
testability. We say that a function property is {\it weakly flexible}
if it is closed under those flexings that do not change the integral
of the function. A good example for this is the property which
consists of those functions whose integral is $1/2$. One can modify
the proof of Theorem \ref{FLEX-TEST} to show that every weakly
flexible function property is testable.
\end{remark}

\begin{proof}
The second assertion is an immediate consequence of the first, so we
only prove (a).

Assume that $d_{\square}(W_n,\RR)\to 0$ but $d_1(W_n,\RR)\geq\eps$
for some fixed $\eps>0$. We can assume that $W_n$ converges to some
$W\in\RR$ in the $\|.\|_{\square}$ norm. Let $S_0=W^{-1}(0)$ ,
$S_1=W^{-1}(1)$ and let $Z_n\in\WW_0$ denote the function which is
$1$ on $S_1$, $0$ on $S_0$ and is identical with $W_n$ anywhere else.
By flexibility, we have $Z_n\in\RR$. By Lemma \ref{WSTAR},
\[
\lim_{n\to\infty}\| W_n-Z_n\|_1=\lim_{n\to\infty}
\Bigl(\int_{S_0}W_n+\int_{S_1}(1-W_n)\Bigr)=\int_{S_0}W+\int_{S_1}(1-W)=0,
\]
which is a contradiction.
\end{proof}

The theorem of Alon and Shapira \cite{AS} follows easily. Let $\PP$
be a hereditary graph property. By Proposition \ref{PROP:GRAPH-FLEX}
$\PP$ is flexible, and so by Theorem \ref{FLEX-TEST}, its closure is
testable. So it suffices to prove that $\PP$ is robust, which follows
from Proposition \ref{G2WG}(b).

Let $U,W\in\WW_0$ and consider a convex combination $Z=\alpha
U+(1-\alpha)W$, $0<\alpha<1$. Then both $U$ and $W$ are flexings
of $Z$. This implies the following very useful observation:

\begin{prop}\label{PP-CONV}
If $\RR$ is a flexible function property, then $\WW_0\setminus \RR$
is a convex set. $\square$
\end{prop}

\begin{corollary}\label{DIST-CONCAVE}
The distance $d_1(U,\RR)$ from a flexible function property is a
concave function of $U$. $\square$
\end{corollary}

Alon and Stav \cite{ASt} proved the surprising fact that for every
hereditary graph property $\PP$ there is a number $p$, $0\le p\le 1$,
such that the random graph $\Ge(n,p)$ is, asymptotically, at maximum
edit distance from the property:

\begin{theorem}[Alon and Stav]\label{ALON-STAV}
For every hereditary graph property $\PP$ there is a number $p$,
$0\le p\le 1$, such that for every graph $G$ with $|V(G)|=n$,
\[
d_1(G,\PP)\le \E(d_1(\Ge(n,p),\PP)) + o(1) \qquad (n\to\infty).
\]
\end{theorem}

The following theorem states a functional version and a
generalization of this fact.

\begin{theorem}\label{ALON-STAV-FN}
{\rm (a)} For every flexible graphon property $\RR$, the maximum of
$d_1(.,\RR)$ is attained by a constant function.

\smallskip

{\rm (b)} For every flexible and robust graph property $\PP$ there is
a number $p$, $0\le p\le 1$, such that for every graph $G$ with
$|V(G)|=n$,
\[
d_1(G,\PP)\le \E(d_1(\Ge(n,p),\PP)) + o(1) \qquad (n\to\infty).
\]
\end{theorem}

By Proposition \ref{PROP:GRAPH-FLEX}(d), Theorem \ref{ALON-STAV} is a
consequence. A further corollary is that the conclusion of Theorem
\ref{ALON-STAV-FN} holds for the properties in Example
\ref{ex:FLEX-GR}(a),(b).

\begin{proof}
Part (a) of the theorem is an easy consequence of Proposition
\ref{PP-CONV} and Corollary \ref{DIST-CONCAVE}. (The proof in fact
works for any norm on $\WW$ instead of the norm $\|.\|_1$.)

To prove part (b), let $p\in[0,1]$ maximize $d_1(p,\overline\PP)$
(where $p$ also denotes the constant $p$ function), and let $b$
denote the maximum value. It suffices to prove that
\begin{equation}\label{eq:FL1}
\E(d_1(\Ge(n,p),\PP))\ge b + o(1) \qquad (n\to\infty),
\end{equation}
and
\begin{equation}\label{eq:FL2}
d_1(G,\PP) \le b + o(1)\qquad (|V(G)|\to\infty).
\end{equation}

The first bound is quite easy: We can fix a choice $G_n$ for the
random graphs so that $d_1(G_n,\PP) \le \E(d_1(\Ge(n,p),\PP))+o(1)$
and $G_n\to p$. Then by Proposition \ref{prop:fura1}, we have
\[
\liminf d_1(G_n,\PP) \ge d_1(p,\overline{\PP}) = b.
\]

Suppose that \eqref{eq:FL2} fails. Then there exists a sequence of
graphs $G_n$ with $|V(G_n)|\to\infty$ such that
\[
d_1(G_n,\PP)\to b'>b.
\]
We may assume that $G_n\to W\in\WW_0$. By part (a),
\[
d_1(W,\overline\PP) \le d_1(p,\overline\PP).
\]
Since the property $\PP$ is robust and flexible, it is testable, and
hence by Proposition \ref{G2WG}(b), we have $d_1(G_n,\PP)
-d_1(W_{G_n},\overline{\PP})\to 0$. Thus
$d_1(W_{G_n},\overline{\PP})\to b'$. Theorem \ref{PROPTEST-FN-DIST}
then implies that $d_1(W,\overline{\PP})=b'$, a contradiction.
\end{proof}

\section{Concluding remarks}

We have mentioned after the definition of testable properties that
some modifications in the definition do not change the notion of
testability. Let us discuss some other possible modifications that
would lead to a different, generally less interesting notion.

Examples \ref{TEST-GR-EX}(b) and (c) suggest that in the definition
of of the edit distance, we could allow adding or removing nodes as
well as adding or removing edges. This would of course change which
graph properties are testable, but would not change the closure of
testable properties, due to Theorem \ref{COMBCHAR}.

Example \ref{TEST-GR-EX}(d) is counterintuitive again, since the
densities in the definition can be estimated from samples easily. The
trouble is that a small error in these densities only implies that
the graph is close to a quasirandom graph in the $d_\square$
distance, not in $d_1$.

Considering how useful the cut distance has been, why don't we
measure the error in cut distance? The answer is that this would lead
to a trivial notion: {\it Every graphon property is testable if we
measure the error in the distance $d_\square$}. Indeed, given a
graphon property $\RR$, we can define $\RR'$ as the property of the
graph $G$ that $d_\square(W_G,\RR)\le 10/\sqrt{\log |V(G)|}$. If
$W\in\RR$, then $G=\Ge(n,W)$ satisfies $\delta_\square(W_G,W)\le
10/\sqrt{\log |V(G)|}$ with large probability, and in such cases
$G\in\RR'$. Conversely, if $d_\square(W,\RR)>\eps$, then for
$G=\Ge(k,W)$ with $k>2^{8/\eps^2}$ we have
$\delta_\square(W_G,W)<\eps/2$ with large probability; if this
happens, then we must have $d_\square(W_G,\RR)>\eps/2$ (else it would
follow that $d_\square(W,\RR)<\eps$), and so if $k$ is large enough,
we have $G\in\RR'$.

We get a different and potentially interesting notion of testing
function properties if instead of $\Ge(n,W)$, we consider the
edge-weighted graph $\Ha(n,W)$ in which we select $n$ random points
$X_1,\dots,X_n$ from $[0,1]$, but instead of randomizing a second
time to get the edges, we keep $W(X_i,X_j)$ as the weight of the edge
$ij$. The singleton set $\RR=\{U\equiv 1/2\}$ would become testable
in this notion. We don't have a characterization of testability in
this sense.

\section*{Acknowledgement}

We are indebted to the anonymous referee for pointing out several
errors and inconsistencies, and for suggesting many improvements to
the presentation.

\end{document}